\documentclass{amsart}
\usepackage{graphicx}
\usepackage{dsfont}
\usepackage{fancyhdr}
\usepackage{helvet}
\usepackage[pdftex,dvipsnames,usenames]{color}
\usepackage[lined,algonl,boxed]{algorithm2e}
\usepackage{amssymb,amsfonts,amsmath,phonetic}
\usepackage{ulem}
\newtheorem{thm}{Theorem}[section]

\newtheorem{lem}[thm]{Lemma}
\newtheorem{prop}[thm]{Proposition}
\theoremstyle{definition}
\newtheorem{defn}[thm]{Definition}
\theoremstyle{remark}
\newtheorem{rem}[thm]{Remark}

\numberwithin{equation}{section}

\newcommand{\abs}[1]{\left\vert#1\right\vert}
\newcommand{\babs}[1]{\Big \vert#1 \Big \vert}
\newcommand{\set}[1]{\left\{#1\right\}}
\newcommand{\parr}[1]{\left (#1\right )}

\newcommand{\eps}{\varepsilon}

\newcommand{\too}{\rightarrow}

\newcommand{\OO}{\mathcal{O}}
\newcommand{\bbar}[1]{\overline{#1}}
\newcommand{\wt}[1]{\widetilde{#1}} 
\newcommand{\wh}[1]{\widehat{#1}} 

\def \M{\mathcal{M}}
\def \N{\mathcal{N}}

\def \z{{z_0}} 
\def \w{{w_0}}

\def \bfi{\textbf{\footnotesize{i}}} 

\def \C{\mathbb{C}} 
\def \D{\mathcal{D}} 
\def \Md{M_{\mathcal{D}}} 

\def \vol{\mbox{\rm{\footnotesize{vol}}}} 
\def \Vol{\mbox{\rm{vol}}} 

\def \bg{\mbox{\boldmath{$g$}}}
\def \bh{\mbox{\boldmath{$h$}}}

\def \R{\mathbb{R}}
\def \ddelta{\mbox{\boldmath{$\delta$}}}

\def \Mbs{M\"{o}bius }
\def \d{\mbox{\bf{d}}}

\newcommand{\ing}[1]{{\color{black}#1}}     
\newcommand{\ingg}[1]{{\color{black}#1}}     
\definecolor{darkred}{rgb}{0.6,0,0}
\begin{document}
\title[{\ingg{Conformal Wasserstein Distances: }}Comparing Surfaces in Polynomial Time]{{\ingg{Conformal Wasserstein Distances:}} Comparing Surfaces in Polynomial Time}
\author{Y. Lipman, I. Daubechies }
\address{Princeton University}
%
%
\begin{abstract}
We present a constructive approach to surface comparison realizable
by a polynomial-time algorithm. We determine the ``similarity''  of
two given surfaces by solving a mass-transportation problem between
their conformal densities. This mass transportation problem differs
from the standard case in that we require the solution to be
invariant under global M\"{o}bius transformations. {\ingg{We present
in detail the case where the surfaces to compare are disk-like; we
also sketch how the approach can be generalized to other types of surfaces.}}
\end{abstract}
\maketitle

\section{introduction}
Alignment and comparison of surfaces (2-manifolds) play a central
role in a wide range of scientific disciplines; they often constitute a
crucial step in a variety of problems in medicine and biology.

Mathematically, the algorithmic problem of surface alignment amounts
to defining a metric function $\d(\cdot,\cdot)$ in the space of
Riemannian 2-manifolds \ing{with the following two properties: 1)} for any two surfaces $\M$ and $\N$,
 $\d(\M,\N)=0$ implies that $\M$ \ing{and} $\N$ \ing{are} isometric, and 2) \ing{given a reasonably large number of reasonably well-distributed sample points on both surfaces, an accurate approximation to} the distance
$\d(\M,\N)$ can be calculated in a \ing{time that grows only} polynomial\ing{ly}
in the sample set size. This second requirement is
crucial to ensure that the algorithm can be used effectively in
applications.

A prominent mathematical approach to define distances between
surfaces that has been proposed for practical applications
\cite{memoli05,BBK06} is the {\em Gromov-Hausdorff (GH) distance};
it considers the surfaces as special cases of \emph{metric spaces}.
T\ing{o determine the} GH distance between \ing{the} metric spaces
$X$ and $Y$\ing{, one} examin\ing{es} all the isometric embeddings
of $X$ and $Y$ into (other) metric spaces; although this distance
possesses many attractive mathematical properties, it is inherently
hard computationally. For instance, computing the $L_p$ version of
the GH distance \ing{between two surfaces} is equivalent to a
non-convex quadratic programming problem, generally over the
integers \cite{memoli07}. This problem is equivalent to integer
quadratic assignment, and is thus NP-hard \cite{Cela98}. In
\cite{memoli07}, Memoli generalizes the GH distance of
\cite{memoli05} by introducing a quadratic mass transportation
scheme to be applied to metric spaces that are also equipped with a
measure (mm spaces); the computation of this {\em Gromov-Wasserstein
(GW) distance for mm spaces} is somewhat easier and more stable to
implement than the original GH distance. The computation of the GW
distance between two surfaces described in \cite{memoli07} utilizes
a (continuous rather than integer) quadratic programming method;
the functional to be minimized is generally not convex and
optimization methods are likely to find local minima rather than the
global minima that realizes the surfaces' distance.

In this paper we propose a new surface alignment procedure,
introducing the {\em conformal Wasserstein distance}.   Our
construction consists in ``geometrically'' aligning the surfaces,
based on uniformization theory and optimal mass transportation. The
uniformization theory serves as a ``dimensionality reduction'' tool,
representing e.g. \ing{a} disk-type surface by its conformal factor
on the unit disk: the corresponding automorphism group (the
disk-preserving M\"{o}bius group) has only three degrees of freedom
and is therefore searchable in polynomial time. Next, the
Kantorovich mass-transportation \cite{Kantorovich1942} is used to
construct a linear functional the minimizer of which furnishes a
metric; as is well-known \cite{Rubner2000-TEM, Villani:2003}, this
can solved by a linear program, and can thus be
computed/approximated in polynomial time as-well.

As far as we know, prior to our work, no polynomial time algorithm
was known to compute, either exactly or up to a good approximation,
the GH distance or any other proposed intrinsic geometric
distances between surfaces. Although  \cite{memoli05} uses a mass
transportation as well (albeit quadratic mass transportation), our
approach is nevertheless different.
We solve the ``standard'' (and thus \emph{linear}) Kantorovich
mass transportation problem, which is convex (even linear) and
solvable via a linear programming method.

There exist earlier papers on aligning or comparing surfaces that
use uniformization. In particular, the papers by Zeng et al.
\cite{Gu2008_a,Gu2008_b} which build upon the work of Gu and Yau
\cite{Gu08}, also use uniformization for surface alignment (albeit
without defining a distance between surfaces). However, they use
prescribed feature points (defined either by the user or by extra
texture information) to calculate an interpolating harmonic map
between the uniformization spaces, and then define the final
correspondence as a composition of the uniformization maps and this
harmonic interpolant. We use only intrinsic geometric information:
we make use of the surfaces' metric (inherited from its embedding
in $\R^3$) and the induced conformal
structure to define deviation from (local) isometry.

Optimal mass transportation has also been used before in
aligning or comparing images. Following the seminal work by Rubner et al.~\cite{Rubner2000-TEM},
it is used extensively in the engineering
literature to define interesting metric distances for images, \ing{interpreted}
as probability densities; in this context the metric is often called
the ''Earth Mover's Distance''.

Our paper is organized as follows: in Section \ref{s:prelim} we
briefly recall some facts about uniformization and optimal mass
transportation that we shall use, at the same time introducing our
notation. Section \ref{s:optimal_vol_trans_for_surfaces} contains
the main results of this paper, constructing the conformal Wasserstein distance metric
between disk-type surfaces, in several steps; we also indicate
how the approach can be generalized to other surfaces. Section
\ref{s:the_discrete_case_implementation} briefly describes the
discrete case.

\section{Background and Notations}
\label{s:prelim}

As described in the introduction, our framework makes use of two
mathematical theories: uniformization theory, to represent the
surfaces as measures defined on a canonical domain, and optimal mass
transportation, to align the measures. In this section we recall
some of their basic properties, and we introduce our notations.

\subsection{Uniformization}
\label{subs:unif} By the celebrated uniformization theory for
Riemann surfaces (see for example \cite{Springer57,Farkas92}), any
simply-connected Riemann surface is conformally equivalent to one of
three canonical domains: the sphere, the complex plane, or the unit
disk. Since every 2-manifold surface $\M$ equipped with a smooth
Riemannian metric $g$ has an induced conformal structure and is thus
a Riemann surface, uniformization applies to such surfaces.
Therefore, every simply-connected surface with a Riemannian metric
can be mapped conformally to one of the three canonical domains
listed above. {\ingg{In this paper, we discuss 2D surfaces, equipped
with a Riemannian metric tensor $g$ (possibly inherited from the
standard 3D metric if the surface is embedded in $\R^3$)}} that have
a finite total volume (i.e. area, since we are delaing with surfaces). 
For convenience, we shall normalize the metric so that the surface area
equals 1.
We shall discuss in detail the case where
the surfaces $\M$ are topologically equivalent to disks. {\ingg{(We
shall address in side remarks how the approach can be
extended to the other cases.)}} For each such $\M$  there exists a
conformal map $\phi:\M \rightarrow \D$, where $\D =\{z \ ; \
|z|<1\}$ is the open unit disk. (we assume that $\M$ does not
include its boundary, if it has one). The map $\phi$ pushes $g$ to a metric on
$\D$; denoting the coordinates in $\D$ by $z=x^1+\bfi x^2$, we can
write this metric as
$$
\wt{g} = \phi_* g = \widetilde{\mu}(z)\, \delta_{ij}\, dx^i \otimes
dx^j,
$$
where $\widetilde{\mu}(z)>0$, Einstein summation convention is used,
and the subscript $*$ denotes the ``push-forward'' action. The
function $\widetilde{\mu}$ can also be viewed as the \emph{density
function} of the measure $\Vol_\M$ induced by the Riemann volume
element: indeed, for (measurable) $A \subset \M$,
\begin{equation}\label{e:volume_element}
    \Vol_\M(A) = \int_{\phi(A)} \widetilde{\mu}(z) \, dx^1\wedge dx^2.
\end{equation}

It will be convenient to use the hyperbolic metric
$(1-|z|^2)^{-2}\delta_{ij} dx^i \otimes dx^j$ as the reference
metric on the unit disk, rather than the standard Euclidean
$\delta_{ij} dx^i \otimes dx^j$; note that the two are conformally
equivalent (with conformal factor $(1-|z|^2)^{-2}$). Instead of the
density $\widetilde{\mu}(z)$, we shall therefore use the {\em
hyperbolic density function}
\begin{equation}\label{e:relation_hyperbolic_euclidean_density}
\mu^H(z):=(1-|z|^2)^{2}\,\widetilde{\mu}(z)\,,
\end{equation}
where the superscript $H$ stands for hyperbolic. We shall often drop
this superscript: unless otherwise stated $\mu=\mu^H$, and
$\nu=\nu^H$. This density function $\mu$
satisfies
$$\Vol_\M(A) = \int_{\phi(A)} \mu(z)\, d\vol_H(z)\,,$$
where $d\vol_H(z)=(1-|z|^2)^{-2}\, dx^1\wedge dx^2$.
{\ingg{In what follows we shall use the symbol $\mu$ both for the
function $\mu^H$ and as a shorthand for the absolutely continuous
measure $\Vol_\M$, and by extension for the surface $\M$ itself.}}

The conformal mappings of $\D$ to itself are the disk-preserving
M\"{o}bius transformations $m \in \Md$, a family with three real
parameters, defined by
\begin{equation}\label{e:disk_mobius}
    m(z) = e^{\bfi \theta}\frac{z-a}{1-\bar{a}z}, \ a\in \D, \ \theta \in [0,2\pi).
\end{equation}
Since these M\"{o}bius transformations satisfy
\begin{equation}\label{e:disk_mobius_is_isometry_of_hyperbolic_geom}
    (1-|m(z)|^2)^{-2}|m'(z)|^2 = (1-|z|^2)^{-2} \,,
\end{equation}
where $m'$ stands for the derivatives of $m$, the pull-back of $\mu$
under a mapping $m\in \Md$ takes on a particularly simple
expression. Setting $w=m(z)$, with $w=y^1+\bfi  y^2$, and
$\widetilde{g}(w)=\widetilde{\mu}(w)\delta_{ij}dy^i\otimes dy^j =
\mu(w) (1-|w|^2)^{-2}\delta_{ij}dy^i\otimes dy^j$, the definition
\[
(m^*\widetilde{g})(z)_{kl}\,dx^{k}\otimes dx^{\ell} :=
\mu(w)\,(1-|w|^2)^{-2}\,\delta_{ij} \, dy^{i}\otimes dy^{j}
\]
implies
\begin{align*}
(m^*\widetilde{g})_{k\ell}(z)\,dx^{k}\otimes dx^{\ell} &=\mu(m(z))(1-|m(z)|^2)^{-2}\,\delta_{ij}\,\frac{\partial y^i}{\partial x^k}\,\frac{\partial y^j}{\partial x^\ell} \,dx^{k}\otimes dx^{\ell}\\
&= \mu(m(z))\,(1-|m(z)|^2)^{-2}\,|m'(z)|^2\,\delta_{k\ell} \,dx^{k}\otimes dx^{\ell}\\
&=\mu(m(z))\,(1-|z|^2)^{-2}\,\delta_{k\ell}\,dx^{k}\otimes dx^{\ell}.\\
\end{align*}
In other words, $(m^*\widetilde{g})(z)_{kl}\,dx^{k}\otimes
dx^{\ell}$ takes on the simple form
$m^*\mu(z)\,(1-|z|^2)^{-2}\,\delta_{kl}\,dx^{k}\otimes dx^{\ell}$,
with
\begin{equation}\label{e:pullback_of_metric_density_mu_by_mobius}
    m^*\mu(z) = \mu(m(z)).
\end{equation}
Likewise, the push-forward, under a disk M\"{o}bius transform
$m(z)=w$, of the (diagonal) Riemannian metric defined by the
density function $\mu=\mu^H$, is again a diagonal metric, with
(hyperbolic) density function $m_{*}\mu (w)=\left(m_{*}\mu
\right)^H(w)$ given by
\begin{equation}\label{e:push_forward_of_metric_density}
    m_* \mu(w) = \mu(m^{-1}(w)).
\end{equation}

It follows that checking whether or not two surfaces $\M$ and $\N$
are isometric, or searching for (near-)\ isometries between $\M$ and
$\N$, is greatly simplified by considering the conformal mappings
from $\M$, $\N$ to $\D$: once the (hyperbolic) density functions
$\mu$ and $\nu$ are known, it suffices to identify $m \in \Md$ such
that $\nu(m(z))$ and $\mu(z)$ coincide (or ``nearly'' coincide, in a
sense to be made precise). This was exploited in
\cite{Lipman:2009:MVF} to construct fast algorithms to find
corresponding points between two given surfaces. In the next
section we provide a precise formalization of this idea using the
notion of {\em optimal mass transportation}, described in the following subsection.

\subsection{Optimal mass transportation}

Optimal mass transportation was introduced by G. Monge
\cite{Monge1781}, and L. Kantorovich \cite{Kantorovich1942}. It
concerns the transformation of one mass distribution into another
while minimizing a cost function that can be viewed as the amount of
work required for the task. In the Kantorovich formulation, to which
we shall stick in this paper, one considers two measure spaces $X,Y$
(each equipped with a $\sigma$-algebra), a probability measure on
each, $\mu \in P(X)$, $\nu \in P(Y)$ (where $P(X),P(Y)$ are the
respective spaces of all probability measures on $X$ and $Y$), and
the space $\Pi(\mu,\nu)$ of probability measures $\pi$ on $X \times
Y$ with marginals $ \mu$ and $\nu$ (resp.), that is, for $A\subset
X$, $B\subset Y$, $\pi(A\times Y) = \mu(A)$ and $\pi(X \times B) =
\nu(B)$. The \emph{optimal} mass transportation is the element  of
$\Pi(\mu,\nu)$ that minimizes $\int_{X \times Y}d(x,y)d\pi(x,y)$,
where $d(x,y)$ is a cost function. (In general, one should consider
an infimum rather than a minimum; in our case, $X$ and $Y$ are
compact, $d(\cdot,\cdot)$ is continuous, and the infimum is
achieved.) The corresponding minimum,
\begin{equation}\label{e:basic_Kantorovich_transporation}
    T_d(\mu,\nu) = \mathop{\inf}_{\pi \in \Pi(\mu,\nu)}\int_{X \times Y}d(x,y)d\pi(x,y),
\end{equation}
is the optimal mass transportation distance between $\mu$ and $\nu$,
with respect to the cost function $d(x,y)$.

Intuitively, one can interpret this as follows: imagine being
confronted with a pile of sand on the one hand (corresponding to
$\mu$), and a hole in the ground on the other hand ($-\nu$), and
assume that the volume of the sand pile equals exactly the volume of
the hole (suitably normalized, $\mu,\nu$ are probability measures).
You wish to fill the hole with the sand from the pile ($\pi \in
\Pi(\mu,\nu)$), in a way that minimizes the amount of work
(represented by $\int d(x,y)d\pi(x,y)$, where $d(\cdot,\cdot)$ can
be thought of as a distance function).

In what follows, we shall apply this framework to the density
functions $\mu$ and $\nu$ on the hyperbolic disk $\D$ obtained by
conformal mappings from two surfaces $\M$, $\N$, as described in the
previous subsection.

The Kantorovich transportation framework cannot be applied
directly to the densities $\mu,\nu$. Indeed, the density $\mu$,
characterizing the Riemannian metric on $\D$ obtained by pushing
forward the metric on $\M$ via the uniformizing map $\phi:\M
\rightarrow \D$, is not uniquely defined: another uniformizing map
$\phi':\M \rightarrow \D$ may well produce a different $\mu'$.
Because the two representations are necessarily isometric
($\phi^{-1} \circ \phi'$ maps $\M$ isometrically to itself), we must
have $\mu'(m(z))=\mu(z)$ for some $m \in \Md$. (In fact, $m=\phi'
\circ \phi^{-1}$.) In a sense, the representation of (disk-type)
surfaces $\M$ as measures over $\D$ should be considered ``modulo''
the disk M\"{o}bius transformations.

We thus need to address how to adapt the optimal transportation
framework to factor out this M\"{o}bius transformation ambiguity.
The next section starts by showing how this can be done.

\section{{\ingg{The conformal Wasserstein framework: o}}ptimal volume transportation for surfaces}
\label{s:optimal_vol_trans_for_surfaces}
We want to measure distances between surfaces by using the
Kantorovich transportation framework to measure the transportation
between the metric densities on $\D$ obtained by uniformization
applied to the surfaces. The main obstacle is that these metric
densities are not uniquely defined; they are defined up to a
M\"{o}bius transformation. In particular, if two densities $\mu$ and
$\nu$ are related by $\nu=m_*\mu$ (i.e. $\mu(z)=\nu(m(z))$), where
$m \in \Md$, then we want our putative distance between $\mu$ and
$\nu$ to be zero, since they describe isometric surfaces, and could
have been obtained by different uniformization maps of the same
surface. We thus want a distance metric between \emph{orbits} of the
group $\Md$ acting on the conformal factors rather than a
metric distance between the conformal factors themselves. If we
choose a metric distance $d$ on $\D$ that is invariant under
M\"{obius} transformations, i.e. that it is a multiple of the
hyperbolic distance on the disk, then a natural definition is as
follows
\begin{equation}
\mbox{\hausaD}(\mu,\nu)=\inf_{m \in \Md} \left( \inf_{\pi \in
\Pi(m_*\mu, \nu)}\,\int_{\D \times \D} d(z,w)\,d\pi(z,w)\,\right).
\label{quot_dist}
\end{equation}
As shown in the Appendix, $\mbox{\hausaD}(\mu,\nu)$ is indeed
a distance between disk-type surfaces; its computation can moreover
be implemented in running times that grow only polynomially in the number
$N$ of sample points used in the discretization of the surface (necessary to
proceed to numerical computation). 
The optimization over $m$ in the definition
of $\mbox{\hausaD}(\mu,\nu)$ always achieves its minimum in some $m$
(depending on $\mu$ and $\nu$ of course); denoting 
this special
minimizing $m \in \Md$ by $m_{\mu,\nu}$, we can rewrite $\mbox{\hausaD}(\mu,\nu)$ 
as the result of a single minimization (for details, see Appendix):
\begin{align} 
\mbox{\hausaD}(\mu,\nu)&=  \inf_{\pi \in \Pi([m_{\mu,\nu}]_*\mu, \nu)}\,\int_{\D \times \D}
d(z,w)\,d\pi(z,w)\nonumber\\
&=\inf_{\pi \in \Pi(\mu, \nu)}\,\int_{\D \times \D}
d(m_{\mu,\nu}(z),w)\,d\pi(z,w)\,.\label{id_Dist}
\end{align}
This is, however, purely formal; since the determination of $m_{\mu,\nu}$ involves the
original double minimization of (\ref{quot_dist}), a numerical implementation
does require solving a mass-transportation functional for many $m \in \Md$.
In practice, this means that, despite its polynomial running time complexity, 
the numerical computation of $\mbox{\hausaD}(\mu,\nu)$ is too heavy for many
applications, in which all pairwise distances must be computed for a 
collection of that may contain hundreds of surfaces \cite{pnas}.
This seems to lead to an impasse, since  there exists no other distance metric
on $\D$ that is conformally invariant, so that the natural ``quotienting operation''
over the group $\Md$ can produce no other metric than $\mbox{\hausaD}(\cdot,\cdot)$.

However, $\mbox{\hausaD}(\mu,\nu)$, rewritten as in (\ref{id_Dist}),
suggests another way in which we can define an appropriate distance metric between orbits of the
group $\Md$ acting on the conformal factors. Note that (\ref{id_Dist}) has exactly the same
form as for a standard
Kantorovich mass transportation scheme, except for the (crucial) difference that
{\em the cost function
depends on $\mu$ and $\nu$}. By retaining the idea of (Kantorovich) mass transportation,
but allowing the use of  cost functions
$d(\cdot,\cdot)$ in the integrand that depend on $\mu$ and $\nu$ 
(without picking them necessarily of the form $d(m_{\mu,\nu}(z),w)$), we can construct
other distance metrics $D$ on the conformal factors that are invariant under action of $\Md$,
i.e. for which $D(\mu, \nu)= D(m_*\mu, \nu)$ for all $m \in \Md$.
In addition, we can pick cost functions of this type
ensuring that the distance between (or dissimilarity of) $\mu$
and $\nu$ exhibits some robustness with respect to
deviations from global isometry. More precisely, we want the
distance to be small for surfaces that are not isometric but
nevertheless very close to isometric on {\em parts} of the surfaces; this
can be achieved by picking a cost function $d^R_{\mu,\nu}(z,w)$
that depends on a comparison of the behavior of $\mu$ and
$\nu$ on {\em neighborhoods} of $z$ and $w$, mapped by $m$ ranging
over $\Md$. This cost function, once incorporated in the
Kantorovich mass transportation framework, will lead to a metric
between disk-type surfaces (some generic conditions aside) based on
solving a \emph{single} mass transportation problem.  The next
subsection shows precisely how this is done. As is the case
throughout the paper, we first give the full details of the construction
for disk-like surfaces, and then indicate later how to generalize
this to e.g. sphere-like surfaces. It is worthwhile to note that the
``quotient approach'' sketched above would not even have been applicable
in a straightforward way to sphere-like surfaces,
since they do not possess a metric invariant under all their
M\"{o}bius transformations. As we shall explain at the end of
this section, the same obstruction will not exist for the construction
introduced in the next subsection.
\subsection{Construction of $d^R_{\mu,\nu}(z,w)$}

We construct $d^R_{\mu,\nu}(z,w)$ so that it indicates the extent
to which a neighborhood of the point $z$ in
$(\D,\mu)$, the (conformal representation of the) first surface,
is isometric with a neighborhood of the point $w$ in $(\D,\nu)$, the (conformal representation of the) second surface. We will need to define two ingredients
for this: the neighborhoods we will use, and how we shall characterize the
(dis)similarity of two neighborhoods, equipped with different metrics.

We start with the neighborhoods.

For a fixed radius $R>0$, we define $\Omega_{z_0,R}$ to be the
hyperbolic geodesic disk of radius $R$ centered at $z_0$. The
following gives an easy procedure to construct these disks. If
$z_0=0$, then  the hyperbolic geodesic disks centered at $z_0=0$ are
also  ``standard'' (i.e. Euclidean) disks centered at 0:
$\Omega_{0,R} = \{z \,;\, |z|\leq r_R \}$, where
$r_R=\mbox{tanh}(R)$. The hyperbolic disks around other centers
are images of these central disks under M\"{o}bius transformations
(= hyperbolic isometries): setting
$m(z)=(z-z_0)(1-z\bar{z_0})^{-1}$, we have
\begin{equation}\label{e:neighborhood_def}
    \Omega_{z_0,R} = m^{-1}(\Omega_{0,R})\,.
\end{equation}
If $m'$, $m''$ are two maps in $\Md$ that both map $z_0$ to 0, then
$m'' \circ (m')^{-1}$
simply rotates $\Omega_{0,R}$ around its center, over some angle $\theta$ determined by $m'$ and $m''$. From this observation one easily checks that
(\ref{e:neighborhood_def})
holds for {\em any} $m \in \Md$ that maps
$z_0$ to $0$. In fact, we have the following more general
\begin{lem}\label{lem:m(omega_z)=omega_w}
For arbitrary $z,w \in \D$ and any $R>0$, every disk preserving M\"{o}bius transformation
$m\in \Md$ that maps $z$ to $w$ (i.e. $w=m(z)$) also maps $\Omega_{z,R}$ to
$\Omega_{w,R}$.
\end{lem}

Next we define how to quantify the (dis)similarity of the pairs
$\left(\Omega_{z_0,R}\,,\, \mu\,\right)$ and $\left(\Omega_{w_0,R}\,,\, \nu\,\right)$.
Since (global) isometries are given by the elements of the
disk-preserving M\"{o}bius group
$\Md$, we will test the extent to which the two patches are isometric
by comparing $\left(\Omega_{w_0,R}\,,\, \nu\,\right)$ with all the images of
$\left(\Omega_{z_0,R}\,,\, \mu\,\right)$ under M\"{o}bius transformations in $\Md$ that take $z_0$ to $w_0$.

To carry out this comparison, we need a norm.
Any metric $g_{ij}(z) dx^i \otimes dx^j$ induces an inner product
on the space of 2-covariant tensors, as follows: if
$\mathbf{a}(z) = a_{ij}(z) \,dx^i \otimes dx^j$
and $\mathbf{b}(z) = b_{ij}(z) \,dx^i \otimes dx^j$
are two 2-covariant tensors in our parameter
space $\D$, then their inner product is defined by
\begin{equation}\label{e:inner_product_2-covariant_tensor}
    \langle \mathbf{a}(z), \mathbf{b}(z)\rangle =
a_{ij}(z)\,b_{k\ell}(z)\,g^{ik}(z)\,g^{j\ell}(z)~;
\end{equation}
as always, this inner product defines a norm,
$\|\mathbf{a}\|_z^2 = a_{ij}(z)\,a_{k\ell}(z)\,g^{ik}(z)\,g^{j\ell}(z)$.
Let us apply this to the computation of the norm of the
difference between the local metric on one surface,
$g_{ij}(z)=\mu(z)(1-|z|^2)^{-2}\delta_{ij}$, and
$h_{ij}(w)=\nu(w)(1-|w|^2)^{-2}\delta_{ij}$, the pull-back metric
from the other surface by a M\"{o}bius transformation $m$. Using
(\ref{e:inner_product_2-covariant_tensor}),
(\ref{e:pullback_of_metric_density_mu_by_mobius}), and writing
$\ddelta$, $\bg$, $\bh$, for the tensors with entries $\delta_{ij}$,  $g_{ij}$, and  
$h_{ij}$, respectively, we have:
\begin{align*}
\|\bg - m^*\bh\|_{z}^2 & = \|\,\mu(z) (1-|z|^2)^{-2}\ddelta -
\nu(m(z))
(1-|z|^2)^{-2} \ddelta\,\|_{z}^2 \\
& = \Big(\mu(z) - \nu(m(z))\Big)^2(1-|z|^2)^{-4}\,\delta_{ij}\,\delta_{k\ell}
\,g^{ik}(z)\,g^{j\ell}(z)=\left(1 - \frac{\nu(m(z))}{\mu(z)}\right)^2.
\end{align*}

{\ing{For every pair of $\mu,\, \nu$, w}}e are now ready to define
the distance function $d^R_{\mu,\nu}(\cdot,\cdot)$ on $\mathcal{D}$:
\begin{equation}\label{e:d_mu,nu(z,w)_def}
    d^R_{\mu,\nu}(z_0,w_0) :=
\mathop{\mathop{\inf}_{m \in \Md\,,}}_{m(z_0)=w_0}\int_{\Omega_{z_0,R}}
\,|\,\mu(z) - (m^*\nu)(z)\,|\, d\vol_H(z),
\end{equation}
where $d\vol_H(z)=(1-|z|^2)^{-2} \,dx\wedge dy$ is the volume form
for the hyperbolic disk. The integral in (\ref{e:d_mu,nu(z,w)_def})
can also be written in the following form, which makes its invariance more
readily apparent:
\begin{equation}\label{e:d_invariant_form}
    \int_{\Omega_{z_0,R}}\left|\,1 - \frac{\nu(m(z))}{\mu(z)}\right| \,d\vol_\M(z)
    = \int_{\Omega_{z_0,R}} \|\bg - m^*\bh \|_z \, d\vol_\M(z),
\end{equation}
where $d\vol_\M(z)=\mu(z)(1-|z|^2)^{-2}\,dx^1\wedge dx^2
=\sqrt{|g_{ij}|}\,dx^1\wedge dx^2$ is the volume form of the first surface $\M$.

The next Lemma shows that although the integration
in (\ref{e:d_invariant_form}) is carried out w.r.t.
the volume of the first surface, this measure of distance is nevertheless symmetric:
\begin{lem}\label{lem:symmetry_of_d_integral}
If  $m\in \Md$ maps $z_0$ to $w_0$, $m(z_0)=w_0$, then
$$
\int_{\Omega_{z_0,R}}\Big|\,\mu(z) - m^*\nu(z) \,\Big|\, d\vol_H(z) =
\int_{\Omega_{w_0,R}}\Big|\,m_*\mu(w) - \nu(w) \, \Big|\, d\vol_H(w).
$$
\end{lem}
\begin{proof}
By the pull-back formula (\ref{e:pullback_of_metric_density_mu_by_mobius}),
we have
$$
\int_{\Omega_{z_0,R}}\Big|\,\mu(z) - m^*\nu(z) \,\Big|\, d\vol_H(z) =
\int_{\Omega_{z_0}}\Big|\,\mu(z) - \nu(m(z)) \,\Big|\, d\vol_H(z).
$$
Performing the change of coordinates $z=m^{-1}(w)$ in
the integral on the right hand side,
we obtain
$$
\int_{m(\Omega_{z_0,R})}\, \Big|\,\mu(m^{-1}(w)) - \nu(w) \Big|\, d\vol_H(w),
$$
where we have used that $m^{-1}$ is an isometry and therefore preserves
the volume element $d\vol_H(w)=(1-|w|^2)^{-2} \,dy^1 \wedge dy^2$.
By Lemma \ref{lem:m(omega_z)=omega_w}, $\,m(\Omega_{z_0,R})=\Omega_{w_0,R}\,$;
using the push-forward formula (\ref{e:push_forward_of_metric_density})
then allows to conclude.
\end{proof}

Note that our point of view in defining our ``distance'' between $z$ and $w$
differs from the classical point of view in mass transportation:
traditionally, $d(z,w)$ is some sort of \emph{physical distance} between the points $z$ and $w$;
in our case $d^R_{\mu,\nu}(z,w)$ measures the dissimilarity of
(neighborhoods of) $z$ and $w$.

The next Theorem lists some important properties of $d^R_{\mu,\nu}$;
its proof is given in the Appendix.
\begin{thm}\label{thm:properties_of_d}
The distance function $d^R_{\mu,\nu}(z,w)$ satisfies the following properties
\begin{table}[ht]
\begin{tabular}{c l l}
{\rm (1)} & $~d^R_{m^*_1\mu,m^*_2\nu}(m^{-1}_1(\z),m^{-1}_2(\w)) = d^R_{\mu,\nu}(\z,\w)~$ & {\rm Invariance under (well-defined)}\\
& & {\rm M\"{o}bius changes of coordinates} \\
&&\\
{\rm (2)} & $~d^R_{\mu,\nu}(\z,\w) = d^R_{\nu,\mu}(\w,\z)~$ & {\rm Symmetry} \\
&&\\
{\rm (3)} & $~d^R_{\mu,\nu}(\z,\w) \geq 0~$ & {\rm Non-negativity} \\
&&\\
{\rm (4)} &\multicolumn{2}{c}{$\!\!\!\!\!\!d^R_{\mu,\nu}(\z,\w) = 0 \,\Longrightarrow \, \Omega_{z_0,R}$ {\rm in} $(\D,\mu)$ {\rm and} $\Omega_{w_0,R}$ {\rm in} $(\D,\nu)$ {\rm are isometric} }\\
&&\\
{\rm (5)} & $~d^R_{m^*\nu, \nu}(m^{-1}(\z),\z)=0~$ & {\rm Reflexivity} \\
&&\\
{\rm (6)} & $~d^R_{\mu_1,\mu_3}(z_1,z_3) \leq d^R_{\mu_1,\mu_2}(z_1,z_2) + d^R_{\mu_2,\mu_3}(z_2,z_3)~$ &
{\rm Triangle inequality}
\end{tabular}
\end{table}

\end{thm}

In addition, the function
$d^R_{\mu,\nu}:\D\times\D\,\rightarrow\,\R$ is continuous. To show
this, we first look a little more closely at the {\ingg{ 1-parameter}} family of disk
M\"{o}bius transformations that map one pre-assigned point $z_0 \in
\D$ to another pre-assigned point $w_0 \in \D$.

\begin{defn}\label{def:M_D,z_0,w_0}
For any pair of points $z_0,\,w_0 \in \D$, we denote by
$M_{D,z_0,w_0}$ the set of M\"{o}bius transformations that map $z_0$
to $w_0$.
\end{defn}

This family of M\"{o}bius transformations is completely
characterized by the following lemma:

\begin{lem}\label{lem:a_and_tet_formula_in_mobius_interpolation}
For any $z_0,w_0 \in \D$, the set $M_{D,z_0,w_0}$ constitutes a
$1$-parameter family of disk M\"{o}bius transformations,
parametrized continuously over $S^1$ (the unit circle). More
precisely, every $m \in M_{D,z_0,w_0}$ is of the form
\begin{equation}\label{e:a_of_mobius}
m(z)= \tau\,\frac{z-a}{1-\overline{a}z}~,~~\mbox{ {\rm{with} }}~~ a
= a(z_0,w_0,\sigma) :=\frac{z_0-w_0
\,\overline{\sigma}}{1-\overline{z_0}\,w_0\,\overline{\sigma}}
~~~\mbox{{\rm and }}~~ \tau = \tau(z_0,w_0,\sigma) := \sigma
\frac{1- \overline{z_0} \,w_0 \,\overline{\sigma}}
{1-z_0\,\overline{w_0}\, \sigma},
\end{equation}
where $\sigma \in S_1:=\{z \in \C\,;\,|z|=1\}$ can be chosen freely.
\end{lem}
\begin{proof}
By (\ref{e:disk_mobius}), the disk M\"{o}bius transformations that
map $z_0$ to $0$ all have the form
\[
m_{\psi,z_0}(z)=e^{\bfi \psi}\,\frac{z-z_0}{1-\overline{z_0}\,z}\,,
~\mbox{ the inverse of which is }~~ m_{\psi,z_0}^{-1}(w)=e^{-\bfi
\psi}\, \frac{w+e^{\bfi \psi}z_0}{1+ e^{-\bfi
\psi}\,\overline{z_0}w}~,
\]
where $\psi \in \R$ can be set arbitrarily. It follows that the
elements of $M_{D,z_0,w_0}$ are given by the family
$m_{\gamma,w_0}^{-1}\circ m_{\psi,z_0}$, with $\psi,\,\gamma \in
\R$. Working this out, one finds that these combinations of
M\"{o}bius transformations take the form (\ref{e:a_of_mobius}), with
$\sigma=e^{\bfi (\psi-\gamma)}$.
\end{proof}
We shall denote by $m_{z_0,w_0,\sigma}$ the special disk M\"{obius}
transformation defined by (\ref{e:a_of_mobius}). In view of our
interest in $d^R_{\mu,\nu}$, we also define the auxiliary function
\begin{eqnarray*}
\Phi&:& \D \times \D \times S_1 \longrightarrow \C\\
&&(z_0,w_0,\sigma) \longmapsto
\int_{\Omega_{z_0,R}}\,|\,\mu(z)-\nu(m_{z_0,w_0,\sigma}(z))\,|\,d\vol_H(z)~.
\end{eqnarray*}
This function has the following continuity properties, inherited
from $\mu$ and $\nu$:

\begin{lem}\label{lem:auxiliary} $~$\\
$\bullet$ For each fixed $(z_0,w_0)$, the function $\Phi(z_0,w_0,\cdot)$ is continuous on $S_1$.\\
$\bullet$ For each fixed $\sigma \in S_1$, $
\Phi(\cdot,\cdot,\sigma) $ is continuous on $\D \times \D$.
Moreover, the family $ \Big(\Phi(\cdot,\cdot,\sigma)\Big)_{\sigma
\in S_1} $ is equicontinuous.
\end{lem}
\begin{proof}
The proof of this Lemma is given in the Appendix.
\end{proof}

Note that since $S^1$ is compact, Lemma \ref{lem:auxiliary} implies
that the infimum in the definition of $d^R_{\mu,\nu}$ can be
replaced by a minimum:
\[
d^R_{\mu,\nu}(z_0,w_0)=\mathop{\min}_{m(z_0)=w_0}\,
\int_{\Omega_{z_0,R}}\,|\,\mu(z)-\nu(m(z))\,|\,d\vol_H(z)~.
\]

We have now all the building blocks to prove
\begin{thm}\label{thm:continuity_of_d^R_mu,nu}
If $\mu$ and $\nu$ are continuous from $\D$ to $\R$, then
$d^R_{\mu,\nu}(z,w)$ is a continuous function on 
$\D\times\D$.
\end{thm}
\begin{proof}
Pick an arbitrary point  $(z_0,w_0) \in \D \times \D$, and pick
$\eps>0$ arbitrarily small.

By Lemma \ref{lem:auxiliary}, there exists a $\delta>0$ such that,
for $|z'_0-z_0|<\delta$, $|w'_0-w_0|<\delta$, we have
\[
\left|\,\Phi(z_0,w_0,\sigma)-\Phi(z'_0,w'_0,\sigma)\,\right|\,\leq\,
\eps~,
\]
uniformly in $\sigma$. Pick now arbitrary $z'_0,w'_0$ so that
$|z_0-z'_0|,|w_0-w'_0|<\delta$.

Let $m_{z_0,w_0,\sigma}$, resp. $m_{z'_0,w'_0,\sigma'}$, be the
minimizing \Mbs transform in the definition of
$d_{\mu,\nu}^R(z_0,w_0)$, resp. $d_{\mu,\nu}^R(z'_0,w'_0)$, i.e.
\[
d^R_{\mu,\nu}(z_0,w_0)= \Phi(z_0,w_0,\sigma) \ \ \textrm{and} \ \
d^R_{\mu,\nu}(z'_0,w'_0)= \Phi(z_0,w_0,\sigma')~.
\]

It then follows that
\begin{align*}
d^R_{\mu,\nu}(z_0,w_0)&=\min_{\tau}\Phi(z_0,w_0,\tau)
\leq \Phi(z_0,w_0,\sigma')\\
&\leq \Phi(z'_0,w'_0,\sigma')+ |\Phi(z_0,w_0,\sigma') -
\Phi(z'_0,w'_0,\sigma')|= d^R_{\mu,\nu}(z'_0,w'_0)+
|\Phi(z_0,w_0,\sigma') - \Phi(z'_0,w'_0,\sigma')|\\
&\leq d^R_{\mu,\nu}(z'_0,w'_0)+ \mathop{\sup}_{\omega \in
S_1}|\Phi(z_0,w_0,\omega) - \Phi(z'_0,w'_0,\omega)| \leq
d^R_{\mu,\nu}(z'_0,w'_0)+ \eps~.
\end{align*}
Likewise $d^R_{\mu,\nu}(z'_0,w'_0) \leq d^R_{\mu,\nu}(z_0,w_0) +
\eps$, so that $\abs{d^R_{\mu,\nu}(z_0,w_0) -
d^R_{\mu,\nu}(z'_0,w'_0)}<\eps$.
\end{proof}

The function $d^R_{\mu,\nu}$ can be extended to a uniformly
continuous function on the closed disk, by using the following
lemma, proved in the Appendix.
\begin{lem}\label{lem:extension_of_d}
Let $\set{(z_k,w_k)}_{k\geq 1} \subset \D\times\D$ be a sequence
that converges, in the Euclidean norm, to some point in $(z',w') \in
\bbar{\D}\times\bbar{\D} \setminus \D\times\D$, that is
$|z_k-z'|+|w_k-w'| \too 0$, as $k \too \infty$. Then, $\lim_{k\too
\infty}d^R_{\xi,\zeta}(z_k,w_k)$ exists and depends only on the
limit point $(z',w')$.
\end{lem}

\subsection{Incorporating $d^R_{\mu,\nu}(z,w)$ into the transportation framework}
The next step in constructing the distance operator between surfaces
is to incorporate the distance $d^R_{\mu,\nu}(z,w)$ defined in the
previous subsection into the (generalized) Kantorovich
transportation model:
\begin{equation}
T^R_d(\mu,\nu)=\inf_{\pi\in \Pi(\mu,\nu)}\int_{\D\times
\D}d^R_{\mu,\nu}(z,w)d\pi(z,w).
\label{e:generalized_Kantorovich_transportation}
\end{equation}
The main result is that this procedure (under some extra conditions)
furnishes a \emph{metric} between (disk-type) surfaces.

\begin{thm}
There exists $\pi^* \in \Pi(\mu,\nu)$ such that
$$\int_{\D\times \D}d^R_{\mu,\nu}(z,w)d\pi^*(z,w)=\inf_{\pi\in \Pi(\mu,\nu)}\int_{\D\times \D}d^R_{\mu,\nu}(z,w)d\pi(z,w).$$
\end{thm}

\begin{proof}
{\ingg{
With a slight abuse of notation, we denote by $\mu$ the probability measure on $\D$
that is absolutely continuous with respect to the hyperbolic measure on $\D$, with
density function equal to the continous function $\mu$ on $\D$. (See subsection \ref{subs:unif}.)
We now define a probability measure $\overline{\mu}$  on $\overline{\D}\,$
by setting $\overline{\mu}(A)=\mu(A\cap\D)$, for arbitrary Borel sets $A \subset \overline{\D}\,$;
$\overline{\nu}$ is defined analogously.
By the Riesz-Markov theorem, the space
of probability measures $\mathcal{P}(\overline{\D} \times \overline{\D})$ can be viewed as
a (closed) subset of the unit ball in $C(\overline{\D}\times\overline{\D})^*$. As such, both
$\mathcal{P}(\overline{\D} \times \overline{\D})$ and its closed subset
$\Pi(\overline{\mu},\overline{\nu})$ are weak$*$-compact,
by the Banach-Alaoglu theorem.
Note that for each
$\overline{\pi} \in \Pi(\overline{\mu},\overline{\nu})$,
we have
\[
\overline{\pi}(\D\times\D)\geq\overline{\pi}(\overline{\D}\times\overline{\D})-\left(\overline{\pi}(\overline{\D}\times[\overline{\D}\setminus\D])+\overline{\pi}([\overline{\D}\setminus\D]\times\overline{\D})\right)\,=\,
1-\overline{\nu}(\overline{\D}\setminus\D)-\overline{\mu}(\overline{\D}\setminus\D)\,=\,1\,,
\]
and thus $\overline{\pi}(\D\times\D)=1$; the restriction $\pi$ of each such $\overline{\pi}$ to the Borel sets contained in $\D\times \D$
is thus a probability measure on $\D\times\D$.\\
Since (the extension to $\overline{\D}\times\overline{\D}$ of) $d_{\mu,\nu}(\cdot,\cdot)$
is an element of $C(\overline{\D}\times\overline{\D})$ by Lemma \ref{lem:extension_of_d},
it follows that the evaluation
$\,\overline{\pi} \mapsto \overline{\pi}(d_{\mu,\nu})\,=\,\int_{\overline{\D}\times\overline{\D}}\,d_{\mu,\nu}(z,w)\,d\overline{\pi}(z,w)\,=\,\int_{\D\times\D}\,d_{\mu,\nu}(z,w)\,d{\pi}(z,w)\,$
is weak$*$-continuous on the weak$*$-compact set $\Pi(\overline{\mu},\overline{\nu})$; it
thus achieves its infimum
in an element $\overline{\pi^*}$ of that set. As observed above,  $\pi^*$, the restriction of $\overline{\pi^*}$ to the Borel sets contained in $\D\times \D$,
is a probability measure on $\D\times\D$, and an element of $\Pi({\mu},{\nu})$; this is the desired minimizer. }}

\end{proof}

Under rather mild conditions, the ``standard'' Kantorovich
transportation (\ref{e:basic_Kantorovich_transporation}) on a metric
spaces $(X,d)$  defines a metric on the space of probability
measures on $X$ . We will prove that our generalization defines a
distance metric as well. More precisely, we shall prove first that
$$
\d^R(\M,\N)=T^R_d(\mu,\nu)
$$
defines a semi-metric in the set of all disk-type surfaces. We shall
restrict ourselves to surfaces that are sufficiently smooth to allow
uniformization, so that they can be globally and conformally
parametrized over the hyperbolic disk. Under some extra
assumptions, we will prove that $\d^R$ is a metric, in the sense
that $\d^R(\M,\N)=0$ implies that $\M$ and $\N$ are isometric.

For the semi-metric part we will again adapt a proof given in
\cite{Villani:2003} to our framework. In particular, we shall make
use of the following ``gluing lemma'':
\begin{lem}
\label{lem:gluing_lemma} Let $\mu_1,\mu_2,\mu_3$ be three
probability measures on $\D$, and let $\pi_{12} \in
\Pi(\mu_1,\mu_2)$, $\pi_{23} \in \Pi(\mu_2,\mu_3)$ be two
transportation plans. Then there exists a probability measure $\pi$
on $\D \times \D \times \D$ that has $\pi_{12},\pi_{23}$ as
marginals, that is $\int_{z_3\in\D} d\pi(z_1,z_2,z_3) =
d\pi_{12}(z_1,z_2) $, and $\int_{z_1\in\D} d\pi(z_1,z_2,z_3) =
d\pi_{23}(z_2,z_3)$.
\end{lem}
This lemma will be used in the proof of the following:

\begin{thm}
For two disk-type surfaces $\M=(\D,\mu)$, $\N=(\D,\nu)$, let
$\d^R(\M,\N)$ be defined by
$$
\d^R(\M,\N)=T^R_d(\mu,\nu).
$$
Then $\d^R$ defines a semi-metric on the space of disk-type
surfaces.
\end{thm}
\begin{proof}

The symmetry of $d^R_{\mu,\nu}$ implies symmetry for $T^R_d$, by the
following argument:
\begin{align*}
T^R_d(\mu,\nu) &=
\mathop{\inf}_{\pi \in \Pi(\mu,\nu)}\int_{\D \times \D}d^R_{\mu,\nu}(z,w)d\pi(z,w) = \mathop{\inf}_{\pi \in \Pi(\mu,\nu)}\int_{\D \times \D}d^R_{\nu,\mu}(w,z)d\pi(z,w) \\
&=
\mathop{\inf}_{\pi \in \Pi(\mu,\nu)}\int_{\D \times \D}d^R_{\nu,\mu}(w,z)d\widetilde{\pi}(w,z),  ~~~~~~~\mbox{ where we have set }~\widetilde{\pi}(w,z)=\pi(z,w)\\
&= T^R_d(\nu,\mu)~.  ~~~~~~~~~(\mbox{ use that }\pi\in\Pi(\mu,\nu)
\Leftrightarrow \widetilde{\pi}\in\Pi(\nu,\mu))
\end{align*}


The non-negativity of $d^R_{\mu,\nu}(\cdot,\cdot)$ automatically
implies $T^R_d(\mu,\nu) \geq 0$.


Next we show that, for any \Mbs transformation $m$,
$T^R_d(\mu,m_*\mu)=0$. To see this, pick the transportation plan
$\pi \in \Pi(\mu,m_*\mu)$ defined by
$$
\int_{\D\times \D} f(z,w) d\pi(z,w) = \int_{\D} f(z,m(z))
\mu(z)\,d\vol_H(z).
$$
On the one hand $\pi \in \Pi(\mu,m_*\mu)$, since
$$\int_{A \times \D} d\pi(z,w) = \int_{A} \mu(z)d\vol_H(z),$$
and
\begin{align*}
\int_{\D\times B} d\pi(z,w) &=
\int_{\D\times \D} \chi_B(w) d\pi(z,w) \\
&=\int_{\D} \chi_B(m(z)) \mu(z)d\vol_H(z) = \int_{\D} \chi_B(w)
\mu_*(w)d\vol_H(w),
\end{align*}
where we used the change of variables $w=m(z)$ in the last step.
Furthermore, $\pi(z,w)$ is concentrated on the graph of $m$, i.e. on
$\set{(z,m(z)) \ ; \ z\in \D} \subset \D \times \D$. Since
$d^R_{\mu,m_*\mu}(z,m(z)) = 0$ for all $z \in \D$ we obtain
therefore $T_d(\mu,m_*\mu) \leq \int_{\D\times \D}
d^R_{\mu,m_*\mu}(z,w)d\pi(z,w)  = 0$.


Finally, we prove the triangle inequality $T^R_d(\mu_1,\mu_3) \leq
T^R_d(\mu_1,\mu_2) + T^R_d(\mu_2,\mu_3)$ . To this end we follow the
argument in the proof given in \cite{Villani:2003} (page 208). This
is where we invoke the gluing Lemma stated above.

We start by picking  arbitrary transportation plans $\pi_{12} \in
\Pi(\mu_1,\mu_2)$ and $\pi_{23} \in \Pi(\mu_2,\mu_3)$. By Lemma
\ref{lem:gluing_lemma} there exists a probability measure $\pi$ on
$\D\times \D \times \D$ with marginals $\pi_{12}$ and $\pi_{23}$.
Denote by $\pi_{13}$ its third marginal, that is
$$\int_{z_2\in \D}d\pi(z_1,z_2,z_3) = d\pi_{13}(z_1,z_3).$$
Then
\begin{align*}
T^R_d(\mu_1,\mu_3) &\leq \int_{\D \times \D} d^R_{\mu_1,\mu_3}(z_1,z_3)d\pi_{13}(z_1,z_3)  = \int_{\D \times \D \times \D} d^R_{\mu_1,\mu_3}(z_1,z_3)d\pi(z_1,z_2,z_3) \\
&\leq \int_{\D \times \D \times \D} \Big( d^R_{\mu_1,\mu_2}(z_1,z_2) + d^R_{\mu_2,\mu_3}(z_2,z_3) \Big )d\pi(z_1,z_2,z_3) \\
&\leq \int_{\D \times \D \times \D} d^R_{\mu_1,\mu_2}(z_1,z_2)
d\pi(z_1,z_2,z_3) +
\int_{\D \times \D \times \D} d^R_{\mu_2,\mu_3}(z_2,z_3) d\pi(z_1,z_2,z_3) \\
&\leq \int_{\D \times \D } d^R_{\mu_1,\mu_2}(z_1,z_2)
d\pi_{12}(z_1,z_2) + \int_{\D \times \D } d^R_{\mu_2,\mu_3}(z_2,z_3)
d\pi_{23}(z_2,z_3),
\end{align*}
where we used the triangle-inequality for $d^R_{\mu,\nu}$ listed in
(Theorem \ref{thm:properties_of_d}). Since we can choose $\pi_{12}$
and $\pi_{23}$ to achieve arbitrary close values to the infimum in
eq.~(\ref{e:generalized_Kantorovich_transportation}) the triangle
inequality follows.
\end{proof}

To qualify as a metric rather than a semi-metric, $\d^R$ (or
$T^R_d$) should be able to distinguish from each other any two
surfaces (or measures) that are not ``identical'', that is
isometric. To prove that they can do so, we need an extra
assumption: we shall require that the surfaces we consider have no
self-isometries. More precisely, we require that each surface $\M$
that we consider satisfies the following definition:

\begin{defn}
A {\ingg{disk-type surface $\M$ is said to be {\em singly $\varrho\mbox{-}_{\mbox{\tiny{H}}}\mbox{fittable}$}
(where
$\varrho \in \R,$ $\varrho > 0$) if, for all $R > \varrho$,
all $z_0\in\D$, and all conformal factors obtained in uniformizations of the disk $\M\,$ there is no
other M\"{o}bius transformation $m$ other
than the identity for which
$$
\int_{\Omega_{z_0,R}} \,
|\mu(z)-\mu(m(z))|\,d\vol_H(z)=0\,.
$$
}}
\end{defn}

\begin{rem}
This definition can also be read as follows: $\M$ is singly
$\varrho\mbox{-}_{\mbox{\tiny{H}}}\mbox{fittable}$ if and only if, for all $R>\varrho$, any two
conformal factors $\mu_1$ and $\mu_2$ for $\M$ satisfy:
\begin{enumerate}
\item
For all $z\in \D$ there exists a unique
minimum to the function $w \mapsto d^R_{\mu_1,\mu_2}(z,w)$.
\item
For all pairs $(z,w)\in \D \times \D$ that achieve this minimum
there exists a unique M\"{o}bius transformation for which the
integral in {\rm(\ref{e:d_mu,nu(z,w)_def})} vanishes (with $\mu_1$
in the role of $\mu$, and $\mu_2$ in that of $\nu$).
\end{enumerate}
Note that in order to ensure that the conditions in the definition hold for all conformal factors, it is sufficient to require that it holds for the conformal factor associated to just one uniformization.
\end{rem}
Essentially, this definition requires that,
from some sufficiently large (hyperbolic) scale onwards, there are
no isometric pieces within $(\D,\mu)$ (or $(\D,\nu)$).

We are ready to  prove the last remaining part of
the main result of this subsection.
We start with a lemma.
\begin{lem}\label{lem:in_every_disk_z_w_d(z,w)=0}
Let $\pi \in \Pi(\mu,\nu)$ be such that $\int_{\D \times
\D}\,d^R_{\mu,\nu}(z,w)\,d\pi(z,w)=0$. Then, for all $z_0 \in \D$
and $\delta >0$, there exists at least one point $z \in
\Omega_{z_0,\delta}$ such that $d^R_{\mu,\nu}(z,w)=0$ for some $w
\in \D$.
\end{lem}
\begin{proof}
By contradiction: assume that there exists a disk
$\Omega_{z_0,\delta}$ such that $d^R_{\mu,\nu}(z,w) >0$ for all
$z\in \Omega_{z_0,\delta}$ and all $w \in \D$. Since
$$
\int_{\Omega(z_0,\delta) \times \D}\,d\pi(z,w) =
\int_{\Omega(z_0,\delta)}\mu(z)\, d\vol_H(z)>0~,
$$
the set $\Omega(z_0,\delta) \times \D$ contains some of the support
of $\pi$. It follows that
$$
\int_{\Omega(z_0,\delta)\times \D} d^R_{\mu,\nu}(z,w)\, d\pi(z,w)>0~,
$$
which contradicts
$$
 \int_{\Omega(z_0,\delta) \times \D}d^R_{\mu,\nu}(z,w)d\pi(z,w)
\le \int_{\D \times \D}d^R_{\mu,\nu}(z,w)d\pi(z,w) = 0~.
$$
\end{proof}

\begin{thm}\label{thm:identity_of_indiscernibles}
Suppose that $\M$ and $\N$ are two surfaces that are singly
$\varrho\mbox{-}_{\mbox{\tiny{H}}}$fittable. If $\d^R(\M,\N)=0$ for some $R > \varrho$, then
there exists a \Mbs transformation $m \in \Md$ that is a global
isometry between $\M=(\D,\mu)$ and $\N=(\D,\nu)$ (where $\mu$ and
$\nu$ are conformal factors of $\M$ and $\N$, respectively).
\end{thm}
\begin{proof}
When $\d^R(\M,\N)=0$, there exists (see \cite{Villani:2003}) $\pi
\in \Pi(\mu,\nu)$ such that
$$
\int_{\D \times \D}d^R_{\mu,\nu}(z,w)d\pi(z,w)=0.
$$

Next, pick an arbitrary point $z_0 \in \D$ such that, for some $w_0
\in \D$, we have $d^R_{\mu,\nu}(z_0,w_0)=0$. (The existence of such
a pair is guaranteed by Lemma \ref{lem:in_every_disk_z_w_d(z,w)=0}.)
This implies that there exists a unique M\"{o}bius transformation
$m_0 \in \Md$ that takes $z_0$ to $w_0$ and that satisfies
$\nu(m_0(z))=\mu(z)$ for all $z \in \Omega_{z_0,R}$. We define
$$
\rho^* = \sup \{\rho \,;\, d^\rho_{\mu,\nu}(z_0,w_0)=0 \};
$$
clearly $\rho^* \geq R$. The theorem will be proved if we show that
$\rho^*= \infty$. We shall do this by contradiction.\\

\begin{figure} [t]
\hspace*{.3 in}
\begin{minipage}{2.8 in}
  \caption{Illustration of the proof of Theorem \ref{thm:identity_of_indiscernibles}}\label{fig:for_proof_identity_of_indiscernibles}
\end{minipage}
\begin{minipage}{3 in}
\vspace*{-1.2 in}

\hspace*{.7 in} \includegraphics[width=2.4 in]{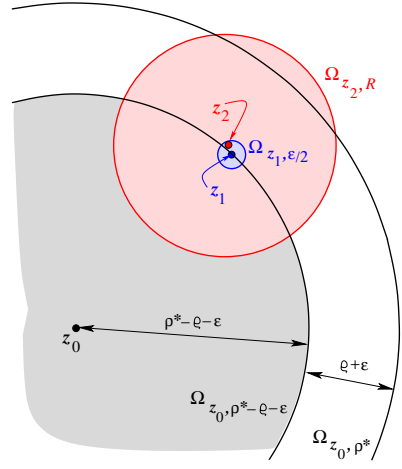}
\end{minipage}
\
\end{figure}

Assume $\rho^* < \infty$. Consider $\Omega_{z_0,\rho^*}$, the
hyperbolic disk around $z_0$ of radius $\rho^*$. (See Figure
\ref{fig:for_proof_identity_of_indiscernibles} for illustration.)
Set $\eps = (R - \varrho)/2$, and consider the points on the
hyperbolic circle $C=\partial \Omega_{z_0, \rho^*-\varrho-\eps}$.
For every $z_1 \in C$, consider the hyperbolic disk
$\Omega_{z_1,\eps/2}$; by  Lemma
\ref{lem:in_every_disk_z_w_d(z,w)=0} there exists a point $z_2$ in
this disk and a corresponding point $w_2 \in \D$ such that
$d^R_{\mu,\nu}(z_2,w_2)=0$, i.e. such that
\[
\int_{\Omega_{z_2,R}}\,|\mu(z)-m'^*\nu(z)|\,d\vol_H(z) \,=\,0~
\]
for some \Mbs transformation $m'$ that maps $z_2$ to $w_2$; in
particular, we have that
\begin{equation}
\mu(z)=\nu(m'(z))~~\mbox{ for all }~ z \in \Omega_{z_2,R}~.
\label{e:ingrid:mprime}
\end{equation}
The hyperbolic distance from $z_2$ to $\partial \Omega_{z_0,\rho^*}$
is at least $\varrho+\eps/2$.
It follows that the hyperbolic disk $\Omega_{z_2,\varrho+\eps/4}$ is
completely contained in $\Omega_{z_0,\rho^*}$; since
$\mu(z)=\nu(m_0(z))$ for all $z \in \Omega_{z_0,\rho^*}$, this must
therefore hold, in particular, for all $z \in
\Omega_{z_2,\varrho+\eps/4}$. Since
$\Omega_{z_2,\varrho+\eps/4}\subset \Omega_{z_2,R}$, we also have
$\mu(z)=\nu(m'(z))$ for all $z \in \Omega_{z_2,\varrho+\eps/4}$, by
(\ref{e:ingrid:mprime}). This implies $\nu(w)=\nu(m_0\circ
(m')^{-1}(w))$ for all $w \in \Omega_{w_2,\varrho+\eps/4}$. Because
$\N$ is singly $\varrho\mbox{-}_{\mbox{\tiny{H}}}$fittable, it follows that $m_0\circ
(m')^{-1}$ must be the identity, or $m_0= m'$. Combining this with
(\ref{e:ingrid:mprime}), we have thus shown that
$\mu(z)=\nu(m_0(z))$ for all $z \in \Omega_{z_2,R}$.

Since the distance between $z_2$ and $z_1$ is at most $\eps/2$, we
also have
\begin{equation*}
\Omega_{z_2,R} \supset \Omega_{z_1,R - \eps/2} = \Omega_{z_1,\varrho
+ 3\eps/2}~. \label{e:ingrid:supset}
\end{equation*}

This implies that if we select such a point $z_2(z_1)$ for each
$z_1\in C$, then $\Omega_{z_0, \rho^* - \varrho-\eps}\cup
\left(\,\cup_{z_1 \in C}\,\Omega_{z_2(z_1),R}\right)$ covers the
open disk $\Omega_{z_0,\rho^*+\eps/2}$. By our earlier argument,
$\mu(z)=\nu(m_0(z))$ for all $z$ in each of the
$\Omega_{z_2(z_1),R}$; since the same is true on $\Omega_{z_0,
\rho^* - \varrho-\eps}$, it follows that $\mu(z)=\nu(m_0(z))$ for
all $z$ in $\Omega_{z_0,\rho^*+\eps/2}$. This contradicts the
definition of $\rho^*$ as the supremum of all radii for which this
was true; it follows that our initial assumption, that $\rho^*$ is
finite, cannot be true, completing the proof.
\end{proof}

For $(\D,\mu)$ to be singly
$\varrho\mbox{-}_{\mbox{\tiny{H}}}$fittable, no two hyperbolic disks
$\Omega_{z,R}$, $\Omega_{w,R}$ (where $w$ can equal $z$) can be
isometric via a M\"{o}bius transformation $m$, if $R > \varrho$,
except if $m=Id$. However, if $z$ is close (in the Euclidean sense)
to the boundary of $\D$, the hyperbolic disk $\Omega_{z,R}$ is very
small in the Euclidean sense, and corresponds to a very small piece
(near the boundary) of $\M$. This means that single
$\varrho\mbox{-}_{\mbox{\tiny{H}}}$fittability imposes restrictions
in increasingly small scales near the boundary of $\M$; from a
practical point of view, this is hard to check, and in many
applications, the behavior of $\M$ close to its boundary is
irrelevant. For this reason, we also formulate the following
relaxation of the results above.

\begin{defn}
A surface $\M$ is said to be {\em singly $A\mbox{-}_{\!_{\M}}$fittable}
(where $A> 0$) if there are no patches (i.e. open, path-connected
sets) in $\M$ of area larger than $A$ that are isometric, with
respect to the metric on $\M$.
\end{defn}

If a surface is singly $A\mbox{-}_{\!_{\M}}$fittable, then it is obviously also
$A'\mbox{-}_{\!_{\M}}$fittable for all $A' \geq A$; the condition of being
$A\mbox{-}_{\!_{\M}}$fittable
becomes more restrictive as $A$ decreases. The following theorem states that two
singly $A\mbox{-}_{\!_{\M}}$fittable surfaces at zero $\d^R$-distance from each other must
necessarily be isometric, up to some small boundary layer.

\begin{thm}
Consider two surfaces $\M$ and $\N$, with corresponding conformal factors $\mu$ and
$\nu$ on $\D$, and suppose $\d^R(\M,\N)=0$ for some $R>0$.
Then the following holds: for arbitrarily large $\rho>0$, there exist
a \Mbs transformation $m \in M_D$ and a value $A>0$ such that
if $\M$ and $\N$
are singly $A\mbox{-}_{\!_{\M}}$fittable
then
$\mu(m(z))=\nu(z)$, for all $z \in \Omega_{0,\rho}$.
\end{thm}
\begin{proof}
Part of the proof follows the same lines as for Theorem
\ref{thm:identity_of_indiscernibles}. We highlight here only the new elements
needed for this proof.

First, note that, for arbitrary $r>0$ and $z_0 \in \D$,
\begin{equation}\label{e:lower_bound_for_patch_area}
    \Vol_\M(\Omega_{z_0,r}) = \int_{\Omega_{z_0,r}}\mu(z)d\vol_H(z)
    \geq \Vol_H(\Omega_{z_0,r})\left[\min_{z\in\Omega_{z_0,r}}\mu(z)\right] =
    \Vol_H(\Omega_{0,r})\left[\min_{z\in\Omega_{z_0,r}}\mu(z)\right].
\end{equation}
This motivates the definition of the sets $\OO_{A,r}$,
\begin{equation}\label{e:O_big_area_set}
    \OO_{A,r} = \set{ z\in \D \ \mid \ \min_{z'\in\Omega_{z,r}}\mu(z') >
    \frac{A}{\Vol_H(0,\Omega_{0,r})}
    };
\end{equation}
$A>0$ is still arbitrary at this point; its value will be set below.

Now pick $r<R$, and set $\eps=(R-r)/2$. Note that if $z \in \OO_{A,r}$, then
$\Vol_{\M}(\Omega_{z,R})\geq\Vol_{\M}(\Omega_{z,r})>A $.

Since $\mu$ is bounded below by a strictly positive constant on each $\Omega_{0,\rho'}$,
we can pick, for arbitrarily large $\rho$, $A>0$ such
that $\Omega_{0,\rho} \subset \OO_{A,r}$; for this it suffices that
$A$ exceed a threshold depending on $\rho$ and $r$.
(Since $\mu(z)\rightarrow 0$
as $z$ approaches the boundary of $\D$ in Euclidean norm, we expect this threshold
to tend towards $0$ as $\rho \rightarrow \infty$.) We assume that
$\Omega_{0,\rho} \subset \OO_{A,r}$ in what follows.

Similar to the proof of Theorem \ref{thm:identity_of_indiscernibles}, we
invoke Lemma
\ref{lem:in_every_disk_z_w_d(z,w)=0} to infer the existence of
$z_0,w_0$ such
that $z_0\in\Omega_{0,\eps/2}$  and
$d^R_{\mu,\nu}(z_0,w_0)=0$.
We denote
$$
\rho^* = \sup \{r' \,;\, d^{r'}_{\mu,\nu}(z_0,w_0)=0 \};
$$
as before, there exists a \Mbs transformation $m$ such that
$\nu(m(z))=\mu(z)$ for all $z$ in $\Omega_{z_0,\rho^*}$.
To complete our proof it therefore suffices to show that $\rho^* \geq \rho + \eps/2$, since
$\Omega_{0,\rho}\subset \Omega_{z_0,\rho+\eps/2}$ .

Suppose the opposite is true, i.e. $\rho^* < \rho+\eps/2$.
By the
same arguments as in the proof of Theorem
\ref{thm:identity_of_indiscernibles},  there exists, for each
$z_1\in \partial \Omega_{z_0,\rho^*-r-\eps}$, a point
$z_2 \in \Omega_{z_1,\eps/2}$ such that $d^R_{\mu,\nu}(z_2,w_2)=0$
for some $w_2$. Since the hyperbolic distance between $z_2$
and $0$ is bounded above by $\eps/2+\rho^*-r-\eps+\eps/2<\rho-r+\eps/2<\rho$,
$z_2 \in \Omega_{0,\rho} \subset \OO_{A,r}$, so that
$\Vol_{\M}(\Omega_{z_2,R})>A $. It then follows from the conditions
on $\M$ and $\N$ that $\nu(m(z))=\mu(z)$ for all z in $\Omega_{z_0,\rho^*}
\cup \Omega_{z_2,R} \supset  \Omega_{z_0,\rho^*}\cup \Omega_{z_1,r+3\eps/2}$.
Repeating the argument for all $z_1\in \partial \Omega_{z_0,\rho^*-r-\eps}$
shows that $\nu(m(z))=\mu(z)$ can be extended to all
$z \in \Omega_{z_0,\rho^*+\eps/2}$, leading to a contradiction that completes the proof.
\end{proof}

So far, we have dealt exclusively with disk-type surfaces.
The approach presented here can also be used for other surfaces,
however. In order to apply this approach to sphere-type
(genus zero) surfaces, for instance, we would need to change only one component in the
construction, namely how to define the neighborhoods $\Omega_{z_0,R}$ in a
M\"{o}bius-invariant way. Since there exists no M\"{o}bius-invariant distance
function on the sphere, we can not define the
neighborhoods $\Omega_{z_0,R}$ as disks with respect to such
an invariant distance. We thus need a different criterium to
pick, among all the circles centered at a point $z_0 \in \M$, the one circle
that shall delimit $\Omega_{z_0,R}$. Since the family of circles
{\em is} invariant under (general) M\"{o}bius transformations, it suffices
to pick a criterium that is itself invariant as well. For our applications,
we pick $\widetilde{\Omega}_{z_0,A}$ to be the interior of the
circle around $z_0$ that has the smallest circumference among all such circles with area
(or volume) $A$ (where $A \in (0,1)\,$).
In the
generic case this procedure defines a unique neighborhood that
can be used in the same way as $\Omega_{z_0,R}$ up till now.

For higher genus (but still homeomorphic) surfaces $\M,\,\N$ it is
natural to use the universal coverings $\wt{\M},\,\wt{\N}$,
respectively.  For genus greater than one, we could use again the
disk-type construction. The fix here (to avoid infinite volume of
the flattened surface via the universal covering) could be to
restrict the compared $z_0\in \D$ to \emph{one copy} of $\M$ in the
hyperbolic disk (and similarly $w_0\in \D$ in $\N$). Treating the
genus one case can be done similarly with similarity transformations
in $\C$.

\section{Discretization and implementation}
\label{s:the_discrete_case_implementation}
Several steps are needed to transform the theoretical framework of
the preceding sections into an algorithm, as described in detail in
\cite{YL-ID-2009}. In a nutshell, the procedure requires three
approximation steps: 1) approximating the smooth surfaces with a
discrete mesh, 2) using discrete conformal theory to construct a
discrete analog of uniformization for meshes, and 3) reducing the
discrete optimization problem (resulting from replacing $\mu,\nu$ in
eq.~(\ref{e:generalized_Kantorovich_transportation}) by their
discrete versions supported on a finite set of points) to a linear
program.
If an equal number of discrete point masses is chosen for the
discrete measure on each of the two surfaces, and all of them are
given equal weight, the corresponding search for the optimal
bistochastic matrix automatically produces a minimizer that is a
{\em permutation}. This means that the minimizer defines a map from
(the discretized version of) one surface to (the discretized version
of) the other.

It follows that the surface distance given in this paper does indeed
lead to a computationally efficient approach, both for finding the
best similarity distance and for identifying the best correspondence
between two (disk-type) surfaces. Figure \ref{f:discrete_type_1}
shows an example of a discrete surface and the corresponding
approximate conformal density visualized as a graph over the unit
disk.

Efficient computation of a distance between surfaces is important
for many applications. As an example, Figure
\ref{fig:distance_graph_embedded} shows an application of our
approach to the characterization of mammals by the surfaces of their
molars \cite{Daubechies10}, comparing high resolution scans of the
masticating surfaces of molars of several lemurs (small primates
living in Madagascar). The figure shows an embedding of eight
molars, coming from individuals in four different species (indicated
by color). The embedding is based on the pairwise distance matrix
($\d^R(\M_i,\M_j)$), and it clearly agrees with the clustering by
species, as communicated to us by the biologists from whom we
obtained the data sets.

\begin{figure}[ht]
\centering \setlength{\tabcolsep}{0.4cm}
\begin{tabular}{@{\hspace{0.0cm}}c@{\hspace{0.2cm}}c@{\hspace{0.0cm}}}
\includegraphics[width=0.4\columnwidth]{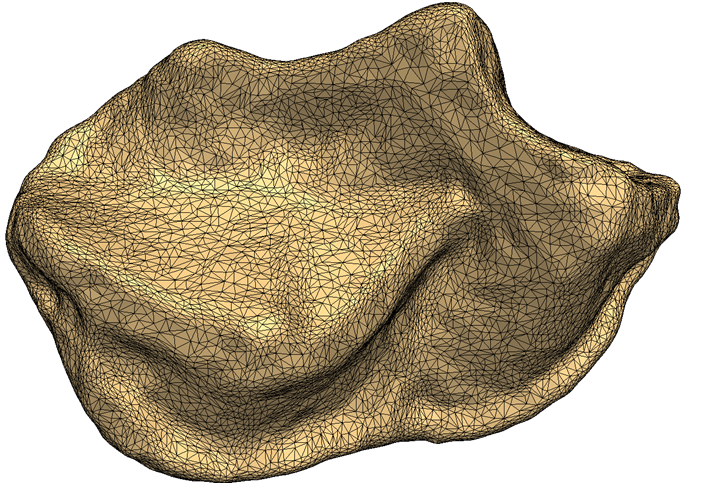}
&
\includegraphics[width=0.5\columnwidth]{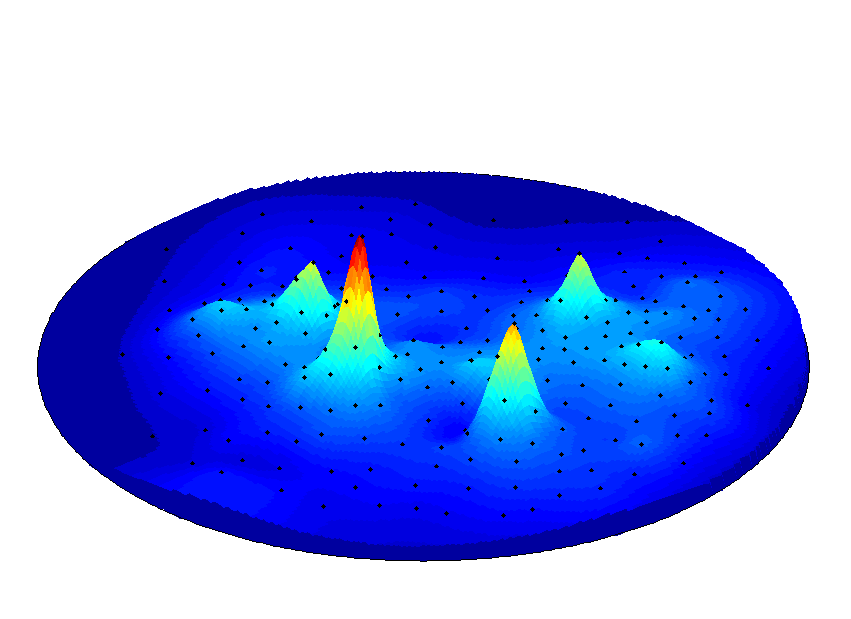}
\\
The discrete representation of a surface (mesh) & Conformal density
over the unit disk
\end{tabular}
\caption{A mammalian tooth discrete surface mesh, and its
approximated conformal factor over the unit
disk.}\label{f:discrete_type_1}
\end{figure}

\begin{figure}[h]
\centering
\includegraphics[width=0.9\columnwidth]{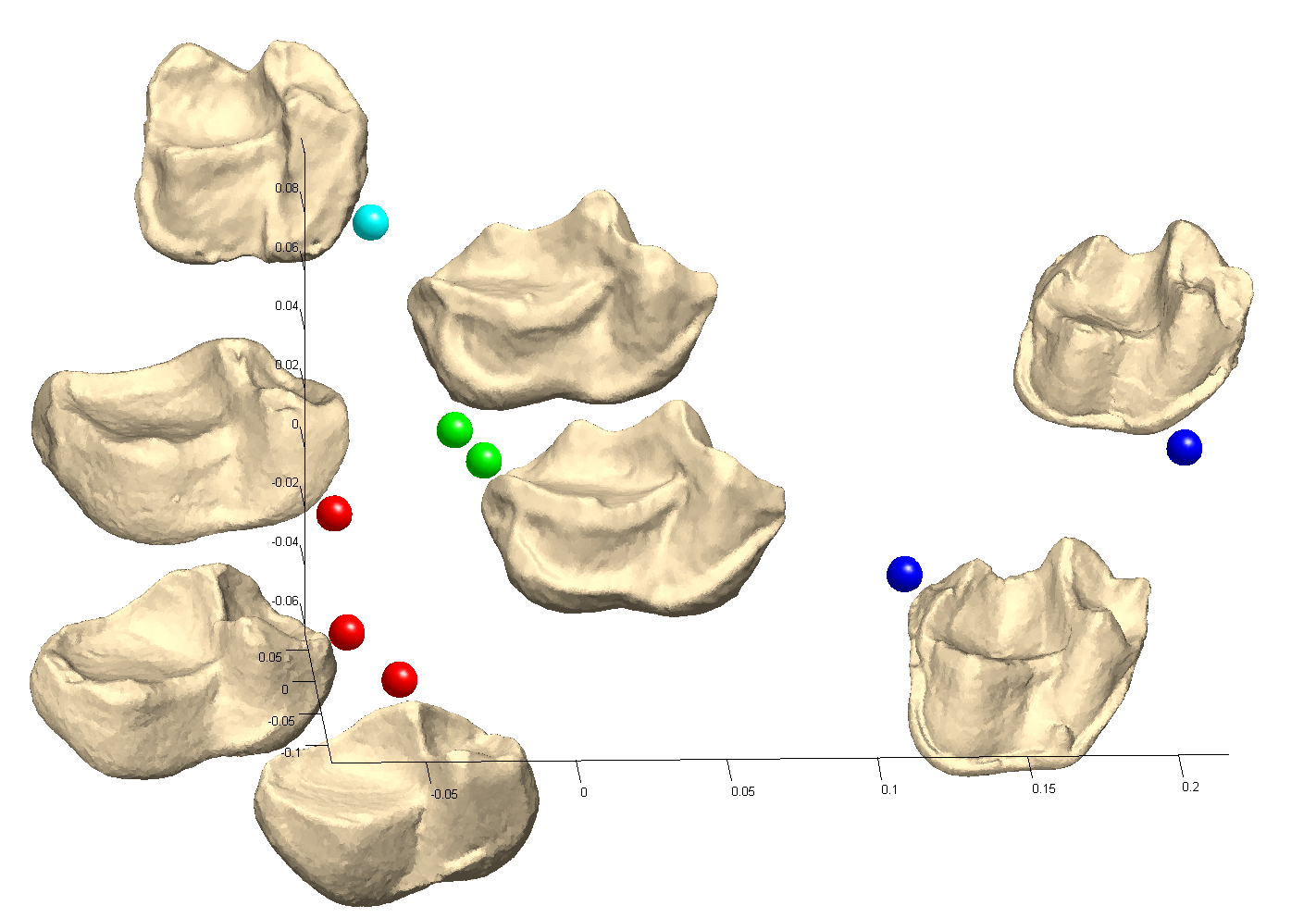}
\caption{Embedding of the distance graph of eight teeth models using
multi-dimensional scaling. Different colors represent different
lemur species. The graph suggests that the geometry of the teeth
might suffice to classify species.}
\label{fig:distance_graph_embedded}
\end{figure}



\section{Acknowledgments}
The authors would like to thank C\'{e}dric Villani and Thomas
Funkhouser for valuable discussions, and Jesus Puente for helping
with the implementation. We are grateful to Jukka Jernvall, Stephen
King, and Doug Boyer for providing us with the tooth data sets, and
for many interesting comments. ID gratefully acknowledges (partial)
support for this work by NSF grant DMS-0914892, and by an AFOSR
Complex Networks grant; YL thanks the Rothschild foundation for
postdoctoral fellowship support.

\appendix
\renewcommand{\thesection}{ \Alph{section}}
\section{}
\label{a:appendix A}
\renewcommand{\thesection}{\Alph{section}}
This Appendix contains some technical proofs of Lemmas and Theorems
stated in section \ref{s:optimal_vol_trans_for_surfaces}, and
\ref{s:the_discrete_case_implementation}.
We start by proving that (\ref{quot_dist}) does indeed define a distance metric
on the family of $\overline{\mu}:=\{m_*\mu\,;\,m \in \Md\,\}$, where the $\mu$ are (smooth)
conformal factors on $\D$, as obtained in sections 2 and 3. We first state a more general lemma:

\begin{lem} Let $X,\,d$ be a metric space, $G$ a group, and $T: \, g\, \mapsto \,T_g $ a representation of $G$ into the 
isometries of $X, \, d$; in particular,  $d$ is invariant under 
the action of the group $G$, i.e. $d(T_gx,T_Gy)\,=\,d(x,y)$, for
all $x,\,y \in X$ and all $g \in G$. Define $\mathcal{C}$ to be the collection of orbits of the representation of $G$, 
i.e. the elements of $\mathcal{C}$ take of the form $\{T_gx\,;\, g \in G\}$, for some $x\in X$. Define $\widetilde{d}$
on $\mathcal{C}\times \mathcal{C}$ by $\widetilde{d}(c_1,c_2)= \inf_{x_1 \in c_1, \, x_2 \in c_2} d(x_1,x_2)$. Then
$\widetilde{d}$ defines a semi-metric on $\mathcal{C}$.
\end{lem}

\begin{proof}
It is obvious that $\widetilde{d}(c_1,c_2)\ge 0$ for all $c_1,\,c_2$ in $\mathcal{C}$; thus only the triangle inequality needs to be established. \\
Since an element $c_1$ of $\mathcal{C}$ can always be written as $c_1\,=\,\{T_gx; g \in G\}$, where $x$ is an arbitrary element of $G$, we obtain, for arbitrary $c_1$, $c_2$, $c_3$ in $\mathcal{C}$,
\begin{align*}
\widetilde{d}(c_1,c_3)&= \inf_{g,\,g' \in G} d(T_gx, T_{g'}z)~~ & \mbox{ where } x\in c_1,\, z \in c_3 \mbox{ are arbitrary}\\
 & \le d(T_{g_1}x, T_{g_3}z)    ~~ & \mbox{ for all } g_1,\, g_3 \in G \\
& \le d(T_{g_1}x, T_{g_2}y)\,+\,d(T_{g_2}y, T_{g_3}z) ~~ & \mbox{ for all } g_1,\, g_2,\,g_3 \in G \mbox{ and all } y \in X\\
& =  d(T_{g_1}x, T_{g_2}y)\,+\,d(T_{g'_2}y, T_{g'_2(g_2)^{-1}g_3}z) ~~ &\mbox{ for all } g_1,\, g_2,\,g'_2,\,g_3 \in G \mbox{ and all } y \in X.
\end{align*}
When $g'_2,\, g_2$ in $G$ are kept fixed, the group elements $g'_2(g_2)^{-1}g_3$ run through all of $G$ as $g_3$ varies 
over $G$. By taking the infimum over the choices of $g_1,\, g_2,\,g'_2,\,g_3 \in G$ in the last expression, we thus obtain
\[
\widetilde{d}(c_1,c_3) \le \widetilde{d}(c_1,c_2)\,+\,\widetilde{d}(c_2,c_3)~,
\]
where $c_2:=\{T_gy\,;\,g \in G\}$. Since $y \in X$ is arbitrary, this proves the triangle inequality in $\mathcal{C}$
for all three-tuples in $\mathcal{C}$.
\end{proof}

Note that one can use the invariance of $d$ under the action of the group on $X$ to define $\widetilde{d}(c_1,c_2)$
via a single  minimization (instead of two): for $x \in c_1$, $y \in c_2$, 
\[
\widetilde{d}(c_1,c_2)\,=\, \inf_{g,\,g' \in G}\,d(T_gx, T_{g'}y)\, =\, \inf_{g,\,g' \in G} \,d(x, T_{g^{-1}g'}y)\,=\,\inf_{g'' \in G}\, d(x, T_{g''}y)~.
\] 

To apply this to  (\ref{quot_dist}), we choose $X$ to be the set of nonnegative $C^1$-functions on $\D$ that
have integral 1 with respect to the hyperbolic area measure on $\D$, and $d$ the Kantorovich mass transport distance between them, with the ``work'' measured in terms of the hyperbolic distance metric $d^{\mbox{\tiny{H}}}$ on $\D$: 
\[
\mbox{d}(\mu,\nu)\,=\, \inf_{\pi \in \Pi(\mu,\nu)}\,\int_{\D\times\D}\,d^{\mbox{\tiny{H}}}(z,w)\, d\pi(z,w)\,.
\]
The group $G$ is here given by $\Md$, and the action of $G$ on $X$ by pull-back: $T_m(\mu)\,=\,m_*\mu$. To apply
the lemma, we first need to establish that :

\begin{lem}
$\mbox{\rm{d}}(\mu,\nu)\,=\,\mbox{\rm{d}}(m_*\mu,m_*\nu)$, for all $\mu$, $\nu$ in $X$, 
and all $m$ in $\Md$.
\end{lem}
\begin{proof}
We first rewrite $\mbox{d}(m_*\mu, m_*\nu)$ in a different way. For each $\pi \in \Pi(\mu,\nu)$, we define the
probability measure $m_*\pi$ on $\D\times \D$ by 
$m_*\pi(E)\,=\,\pi\left(\,\{\,(m^{-1}z,m^{-1}w)\,;\,(z,w) \in E  \,\}\,\right)$. It is straightforward to check that
$m_*\pi(A \times \D)\,=\,m_*\mu(A)$ and $m_*\pi(\D \times B)\,=\,m_*\nu(B)$ for all Borel sets $A$ , $B \subset \D$; thus  $m_*\pi \in \Pi(m_*\mu,m_*\nu)$. One can analogously define $m^*\pi$; again it is straightforward that
$m^*m_*\pi\,=\,\pi$. It follows that $\Pi(m_*\mu,m_*\nu)$ is exactly equal to $\{\,m_*\pi\,;\,\pi \in \Pi(\mu,\nu) \,\}$.

Consequently, using the invariance $d^{\mbox{\tiny{H}}}(mz,mw)\,=\,d^{\mbox{\tiny{H}}}(z,w)$, we obtain
\begin{align*}
\mbox{d}(m_*\mu,m_*\nu)\,&=\,\inf_{\pi \in \Pi(\mu,\nu)}\,\int_{\D \times \D}\,d^{\mbox{\tiny{H}}}(z,w)
\, d(m_*\pi)(z,w)
=\inf_{\pi \in \Pi(\mu,\nu)}\,\int_{\D \times \D}\,d^{\mbox{\tiny{H}}}(mu,mv)\, d\pi(u,v) \\
&=\inf_{\pi \in \Pi(\mu,\nu)}\,\int_{\D \times \D}\,d^{\mbox{\tiny{H}}}(u,v)\, d\pi(u,v) \,=\,\mbox{d}(\mu,\nu)~.
\end{align*}
\end{proof}

It follows that we can indeed apply the first lemma, and that (\ref{quot_dist})
defines a semi-metric on the equivalence classes 
of conformal factors,
where two conformal factors are viewed as equivalent if one can be obtained from the other by pushing it forward (or
backward) through a M\"{o}bius transformation.

It turns out that in this case, the infimum over the choices $m \in \Md$ is in fact always achieved (and is thus a minimum):

\begin{lem} Let $\mu$ and $\nu$ be conformal factors obtained by uniformizing two
smooth disk-type surfaces, with $\mbox{Distance}(\mu,\nu)<\infty$ defined as in {\rm(\ref{quot_dist})}. 
Then there exists a M\"{o}bius transformation $m \in \Md$ such that 
$\mbox{Distance}(\mu,\nu)\,=\, \mbox{\rm{d}}(m_*\mu,\nu)$.
\end{lem}
\begin{proof}
Consider two arbitrary (but fixed) conformal factors $\mu$ and $\nu$ on $\D$.
There exists a sequence 
$\left(m_n\right)_{n \in \mathbb{N}}$ such that 
$\inf_{\pi \in \Pi(\mu,\nu)}\,\int_{\D\times\D}\,d^{\mbox{\tiny{H}}}(m_n(z),w)\,d\pi(z,w)$ 
$\rightarrow \mbox{Distance}(\mu,\nu)\,$ as $n\, \rightarrow\,\infty$.
Each of these $m_n$ can be written in the form given by (\ref{e:disk_mobius}), with corresponding $a_n \in \D$, and 
$e^{i\theta_n} \in \mathbb{T}:=\{\,z\in \mathbb{C}\,;\, |z|=1\,\}$. By passing to a subsequence if necessary, we can assume, without loss of generality, that
the sequences $\left(a_n \right)_{n \in \mathbb{N}}$ and $\left( e^{i \theta_n}\right)_{n \in \mathbb{N}}$ 
converge in  $\overline{\D}$ (the closure of $\D$) and $\mathbb{T}$, respectively,
to limits we denote by $\overline{a}$
and $e^{i\overline{\theta}}$. 

If $\overline{a}$ lies in the open disk $\D$, then it defines, together with  
$e^{i\overline{\theta}}$, a corresponding $\overline{m}\in \Md$. We then have, for all $z,\,w$ in $\D$,
$\lim_{n\rightarrow \infty}\,d^{\mbox{\tiny{H}}}(m_n(z),w)\,=\,d^{\mbox{\tiny{H}}}(\overline{m}(z),w)$.
On the other hand, for sufficiently large $n$ we have
\begin{eqnarray*}
d^{\mbox{\tiny{H}}}(m_n(z),w) &\leq& d^{\mbox{\tiny{H}}}(m_n(z),0)\,+\, d^{\mbox{\tiny{H}}}(0,w) \\
&=& d^{\mbox{\tiny{H}}}(z,a_n)\,+\, d^{\mbox{\tiny{H}}}(0,w)\\
&\leq& d^{\mbox{\tiny{H}}}(z,\overline{a})\,+1\,+\, d^{\mbox{\tiny{H}}}(0,w)\\
&\leq& d^{\mbox{\tiny{H}}}(z,0)\,+\, d^{\mbox{\tiny{H}}}(0,\overline{a})\,+1\,+\, d^{\mbox{\tiny{H}}}(0,w),
\end{eqnarray*}
where we have used the invariance of $d^{\mbox{\tiny{H}}}$ under M\"{o}bius transformations and 
$m_n(a_n)=0$ in the second line, and where we assume $n$ sufficiently large to ensure 
$d^{\mbox{\tiny{H}}}(a_n,\overline{a})\leq 1$ in the third.

Therefore $d^{\mbox{\tiny{H}}}(m_n(z),w)$ is bounded, uniformly in $n$, by a function
that is absolutely
integrable with respect to $\pi$ (by the argument used just before the statement of Lemma A.2); 
the dominated convergence theorem then implies that
\begin{eqnarray*}
\mbox{\hausaD}(\mu,\nu)&=&\lim_{n\rightarrow \infty}\,\int_{\D\times \D}\,d^{\mbox{\tiny{H}}}(m_n(z),w)\,d\pi(z,w)\\
&=&\int_{\D\times \D}\,d^{\mbox{\tiny{H}}}(\overline{m}(z),w)\,d\pi(z,w)\,,
\end{eqnarray*}
so that we are done for the case where $\overline{a} \in \D$.

It remains to discuss the case where $\overline{a} \in \overline{\D}\setminus \D= \mathbb{T}$, i.e. $|a|\,=\,1$.
The proof will be complete if we show that this is impossible; we will establish this by contradiction. 

From now on, we suppose that 
$|\overline{a}|=1$. By the integrability of $\mu$ and $\nu$, we can find an increasing sequence of $\rho_n<1$ such that
$\mu\left(\{z\,;\,1>|z|>\rho_n \, \} \right) < 1/n$ and 
$\nu\left(\{z\,;\,1>|z|>\rho_n\,  \} \right) < 1/n $ .
It is easy to check that 
\[
\mbox{for } |a|>R \mbox{ and }|z|<\rho<R \mbox{ : }\left|\frac{z-a}{1-a^*z}  \right| > 
\frac{R-\rho}{1-R\rho}~.
\]
This lower bound tends to 1 as $R$ tends to 1, regardless of the value of $\rho < 1$. It follows that there
exist $R_n<1$ so that 
\[
\inf_{|z| < \rho_n} \left|\frac{z-a}{1-a^*z}  \right| > (n+\rho_n)/(n+1), ~ \mbox{ for all } a \mbox{ with } |a|>R_n~.
\]
Because
$|\overline{a}|=1$, we can find a  $k_n \in \mathbb{N}$ such that
$|a_k|>R_n$ for all $k > k_n$; consequently 
$|m_k(z)|\,=\,|z-a_k|/|1-(a_k)^*z|>(n+\rho_n)/(n+1)\,$ for all $k> k_n$ and all $z$ with $|z|< \rho_n$. 
It then follows that (with the notation $\D_n\,:=\, \{\,z\,;\,|z|< \rho_n\,\}$)
\begin{align*}
\forall k> k_n\,, \, \forall \pi \in \Pi(\mu,\nu):\,\int_{\D\times \D}\,&d^{\mbox{\tiny{H}}}(m_k(z),w)\,d\pi(z,w)\\
&\geq \int_{\D_n\times \D_n}\,d^{\mbox{\tiny{H}}}(m_k(z),w)\,d\pi(z,w)\\
&\geq \int_{\D_n\times \D_n}\,\left[\,\inf_{|v|<\rho_n,\,|u|>(n+\rho_n)/(n+1))}d^{\mbox{\tiny{H}}}(u,v)\,\right]\,d\pi(z,w)\\
&\geq \frac{\ln(n+1)}{2}\,\int_{\D_n\times \D_n}\,d\pi(z,w) \\
&\geq \frac{\ln(n+1)}{2}\,
\left[\,1\,-\,\pi((\D\setminus\D_n)\times \D)\,-\,\pi(\D\times (\D\setminus\D_n))\,\right]\\
& \geq \frac{\ln(n+1)}{2}\,
\left(\,1\,-\,\frac{2}{n}   \,\right)~,
\end{align*}
where we have used that if $|u|<r<1$, and $|v|>(n+r)/(n+1)$, then $d^{\mbox{\tiny{H}}}(u,v)\geq\int_{r}^{(n+r)/(1+n)}\,
\frac{1}{1-t^2} \,dt \geq \frac{1}{2}\,\int_{r}^{(n+r)/(1+n)}\,\frac{1}{1-t}\,dt \,=\,\frac{\ln (n+1)}{2} $ .
This shows, in particular, that 
\[
\mbox{\rm{d}}(\mu,m_k^*\nu)\,=\,\inf_{\pi \in \Pi(\mu,\nu)}\,\int_{\D\times \D}\,d^{\mbox{\tiny{H}}}(m_k(z),w)\,d\pi(z,w) \geq \frac{\ln (n+1)}{4}
\]
for all $k>k_n$ and $n>4$. This implies that, for arbitrary $n>4,\, n \in \mathbb{N}$, 
\[
\mbox{Distance}(\mu,\nu)\,=\,lim_{k\rightarrow \infty}\,\mbox{\rm{d}}(\mu,m_k^*\nu) \geq \frac{\ln (n+1)}{4}\,,
\]
i.e. $\mbox{Distance}(\mu,\nu)\,=\,\infty$, a contradiction. This finishes the argument that $|\overline{a}|=1$ is not possible, and completes the proof.
\end{proof}

It is now easy to see that $\mbox{Distance}(\mu,\nu)$ defines a true metric on the equivalence classes of conformal factors:

\begin{prop} The $\mbox{Distance}(\mu,\nu)$ defined in {\rm{(\ref{quot_dist})}} is a metric on the set of orbits
$\overline{\mu}:=\{m_*\mu\,;\,m \in \Md\,\}$ of conformal factors under the action of $\Md$.
\end{prop}
\begin{proof}
In view of the first lemma, we need to prove only that $\mbox{Distance}(\mu,\nu)=0$ implies that there exists a M\"{o}bius transformation $m \in \Md$ such that $\nu= m_*\mu$. By the second lemma, we know that
$\mbox{Distance}(\mu,\nu)=\rm{d}(m_*\mu,\nu)$ for some $m \in \Md$. For this
$m$ there exist thus $\pi_k \in \Pi(\mu,\nu)$ such that $\int_{\D\times \D} \, d^{\mbox{\tiny{H}}}(m(z),w)\,d\pi_k(z,w)$
tends to 0 as $k$ tends to $\infty$. By passing to a subsequence if necessary, we can, using the weak$*$-compactness of the set of probability measures on $\D\times\D$, assume that $\pi_k \rightarrow \overline{\pi}$ as $k$ tends to infinity, 
in the weak$*$-topology, where $\overline{\pi}$ is a measure of weight at most 1. 
Since the $\pi_k$ all have marginals $\mu$ and $\nu$, respectively, it follows that $\overline{\pi}$ must have these marginals as well, which guarantees that $\overline{\pi}$ is itself a probability measure and an element of $\Pi(\mu,\nu)$. We have moreover
\[
0 \,=\, \mbox{Distance}(\mu,\nu)\,=\,\int_{\D\times \D} \, d^{\mbox{\tiny{H}}}(m(z),w)\,d\overline{\pi}(z,w)~,
\]
implying that the support of $\overline{\pi}$ is contained in the subset 
$\{\,(z, m(z))\,;\, z \in \D\,\} \subset \D\times\D$. It then follows that, for any Borel set $A \subset \D$,
\[
\int_{A}\,\mu(z)\,d\vol_H(z)\,=\,\overline{\pi}(A\times \D)\,=\,\overline{\pi}(A\times m(A))\,=\,\overline{\pi}(\D\times m(A))\,=\,\int_{A}\,\nu(m(z))\,d\vol_H(z)~.
\]
This is possible for the continuous functions $\mu$ and $\nu$ only if $\nu=m^*\mu$, or equivalently, $\mu=m_*\nu$

\end{proof}

Next we prove the
list of properties of the distance function $d^R_{\mu,\nu}(z,w)$
given in Theorem  \ref{thm:properties_of_d}:

{\bf Theorem} \ref{thm:properties_of_d}

{\em The distance function $d^R_{\mu,\nu}(z,w)$ satisfies the
following properties}
\begin{table}[ht]
\begin{tabular}{c l l}
{\rm (1)} & $~d^R_{m^*_1\mu,m^*_2\nu}(m^{-1}_1(\z),m^{-1}_2(\w)) = d^R_{\mu,\nu}(\z,\w)~$ & {\rm Invariance under (well-defined)}\\
& & {\rm M\"{o}bius changes of coordinates} \\
&&\\
{\rm (2)} & $~d^R_{\mu,\nu}(\z,\w) = d^R_{\nu,\mu}(\w,\z)~$ & {\rm Symmetry} \\
&&\\
{\rm (3)} & $~d^R_{\mu,\nu}(\z,\w) \geq 0~$ & {\rm Non-negativity} \\
&&\\
{\rm (4)} & \multicolumn{2}{c}{$\!\!\!\!\!\!~d^R_{\mu,\nu}(\z,\w) = 0 \,\Longrightarrow \, \Omega_{z_0,R}$ {\rm in} $(\D,\mu)$ {\rm and} $\Omega_{w_0,R}$ {\rm in} $(\D,\nu)$ {\rm are isometric} }\\
&&\\
{\rm (5)} & $~d^R_{m^*\nu, \nu}(m^{-1}(\z),\z)=0~$ & {\rm Reflexivity} \\
&&\\
{\rm (6)} & $~d^R_{\mu_1,\mu_3}(z_1,z_3) \leq
d^R_{\mu_1,\mu_2}(z_1,z_2) + d^R_{\mu_2,\mu_3}(z_2,z_3)~$ & {\rm
Triangle inequality}
\end{tabular}
\end{table}

\begin{proof}

For (1), denote $m_1^{-1}(z_0)=z_1$, and $m_2^{-1}(w_0)=w_1$. Then
\begin{align*}
d^R_{m_1^* \mu, m_2^* \nu}(z_1,w_1) &=
\mathop{\inf}_{m(z_1)=w_1} \int_{\Omega_{z_1,R}}|m_1^*\mu(z)-m^*m_2^*\nu(z)|d\vol_H(z) \\
&= \mathop{\inf}_{m(z_1)=w_1}
\int_{\Omega_{z_1,R}}|\mu(m_1(z))-\nu(m_2(m(z)))|d\vol_H(z).
\end{align*}
Next set $\widetilde{m} = m_2 \circ m \circ m_1^{-1}$. Note that
$\widetilde{m}(z_0)=w_0$. Plugging $m_2(m(z))=\widetilde{m}(m_1(z))$
into the integral and carrying out the change of variables
$m_1(z)=z'\,$, we obtain
$$
\mathop{\inf}_{m(z_1)=w_1} \int_{\Omega_{z_1,R}}
|\,\mu(z')-\nu(\widetilde{m}(z'))\,|d\vol_H(z') =
\mathop{\inf}_{\widetilde{m}(z_0)=w_0}
\int_{\Omega_{z_0,R}}|\,\mu(z')-\nu(\widetilde{m}(z'))\,|\,d\vol_H(z').
$$


For (2), we use Lemma \ref{lem:symmetry_of_d_integral} and equations
(\ref{e:pullback_of_metric_density_mu_by_mobius}),
(\ref{e:push_forward_of_metric_density}) to write
\begin{align*}
d^R_{\mu,\nu}(z_0,w_0)&=
\mathop{\inf}_{m(z_0)=w_0} \int_{\Omega_\z,R} |\mu(z)-m^*\nu(z)|d\vol_H(z) \\
&= \mathop{\inf}_{m(z_0)=w_0} \int_{\Omega_\w,R}
|(m^{-1})^*\mu(w)-\nu(w)|d\vol_H(w) = d^R_{\nu,\mu}(w_0,z_0).
\end{align*}


(3) and (4) are immediate from the definition of $d^R_{\mu,\nu}$.


(5) follows from the observation that the minimizing $m$ (in the
definition (\ref{e:d_mu,nu(z,w)_def}) of $d^R_{\mu,\nu}$) is $m_1$
itself, for which the integrand, and thus the whole integral
vanishes identically.


For (6), let $m_1$ be a M\"{o}bius transformation such that
$m_1(z_1)=z_2$, and $m_2$ such that $m_2(z_2)=z_3$. Setting $m=m_2
\circ m_1$, we have
\begin{align}
d^R_{\mu_1,\mu_3}(z_1,z_3) &\leq \int_{\Omega_{z_1,R}}
|\,\mu_1(z) - m^* \mu_3(z)\,|\,d\vol_H(z) \nonumber\\
&\leq \int_{\Omega_{z_1,R}}|\,\mu_1(z) - m_1^*
\mu_2(z)\,|\,d\vol_H(z) + \int_{\Omega_{z_1,R}}|\,m_1^* \mu_2(z) -
m^* \mu_3(z)\,|\,d\vol_H(z)\,. \label{intermed}
\end{align}

The second term in (\ref{intermed}) can be rewritten as (using
Lemma \ref{lem:symmetry_of_d_integral}, the change of coordinates
$m_1(z_1)=z_2$ and the observation $m^*=m_1^* m_2^*$)
\begin{align*}
\int_{\Omega_{z_1,R}}|\,m_1^* \mu_2(z) - m^* \mu_3(z)\,|\,d\vol_H(z)
&=\int_{\Omega_{z_2,R}} |\,m_{1*}m_1^*\mu_2(w) - m_{1*}m_1^* m_2^*
\mu_3(w)\,|\,
d\vol_H(w)\\
&=\int_{\Omega_{z_2,R}} |\,\mu_2(w) - m_2^* \mu_3(w)\,|\,d\vol_H(w).
\end{align*}
We have thus
\[
d^R_{\mu_1,\mu_3}(z_1,z_3) \leq \int_{\Omega_{z_1,R}}|\,\mu_1(z) -
m_1^* \mu_2(z)\,|\,d\vol_H(z) + \int_{\Omega_{z_2,R}} |\,\mu_2(w) -
m_2^* \mu_3(w)\,|\,d\vol_H(w)~,
\]
and this for any $m_1,\,m_2 \in \Md$ such that $m_1(z_1)=z_2$ and
$m_2(z_2)=z_3$. Minimizing over $m_1$ and $m_2$ then leads to the
desired result.

\end{proof}

Next we prove the continuity properties of the function
$\Phi(z_0,w_0,\sigma) =
\int_{\Omega(z_0,R)}\,|\,\mu(z)-\nu(m_{z_0,w_0,\sigma}(z))\,|\,d\vol_H(z)$,
stated in Lemma \ref{lem:auxiliary}, which were used to prove
continuity of $d^R_{\mu,\nu}$ itself (in Theorem
\ref{thm:continuity_of_d^R_mu,nu}).

{\bf Lemma \ref{lem:auxiliary}} \\
{\em $\bullet$ For each fixed $(z_0,w_0)$ the function $\Phi(z_0,w_0,\cdot)$ is continuous on $S_1$.\\
$\bullet$ For each fixed $\sigma \in S_1$, $
\Phi(\cdot,\cdot,\sigma) $ is continuous on $\D \times \D$.
Moreover, the family $ \Big(\Phi(\cdot,\cdot,\sigma)\Big)_{\sigma
\in S_1} $ is equicontinuous.}

\begin{proof}
We start with the continuity in $\sigma$. We have
\[
\left|\,\Phi(z_0,w_0,\sigma)-\Phi(z_0,w_0,\sigma')\,\right| \leq
\int_{\Omega(z_0,R)}\,|\nu(m_{z_0,w_0,\sigma}(z))-\nu(m_{z_0,w_0,\sigma'}(z))\,|
\,d\vol_H(z)~.
\]
Because $\nu$ is continuous on $\D$, its restriction to the compact
set $\overline{\Omega(w_0,R)}$ (the closure of $\Omega(w_0,R)$) is
bounded. Since the hyperbolic volume of $\Omega(z_0,R)$ is finite,
the integrand is dominated, uniformly in $\sigma'$, by an integrable
function.  Since $m_{z_0,w_0,\sigma}(z)$ is obviously continuous in
$\sigma$, we can use the dominated convergence theorem to conclude.

Since $S^1$ is compact, this continuity implies that the infimum in
the definition of $d^R_{\mu,\nu}$ can be replaced by a minimum:
\[
d^R_{\mu,\nu}(z_0,w_0)=\mathop{\min}_{m(z_0)=w_0}\,
\int_{\Omega(z_0,R)}\,|\,\mu(z)-\nu(m(z))\,|\,d\vol_H(z)~.
\]

Next we prove continuity in $z_0$ and $w_0$ (with estimates that are
uniform in $\sigma$).

Consider two pairs of points, $(z_0,w_0)$ and $(z'_0,w'_0) \in \D
\times \D$. Then
\begin{align}
&|\,\Phi(z_0,w_0,\sigma)-\Phi(z'_0,w'_0,\sigma)\,|\nonumber\\
&~~~~~=\,\left|\,
\int_{\Omega(z_0,R)}\,|\,\mu(z)-\nu(m_{z_0,w_0,\sigma}(z))\,|\,d\vol_H(z)
- \int_{\Omega(z'_0,R)}\,|\,\mu(u)-\nu(m_{z'_0,w'_0 \sigma}(u)\,|\,d\vol_H(u)\,\right|\nonumber\\
&~~~~~\leq\,
\left|\,\int_{\Omega(z_0,R)}\,|\,\mu(z)-\nu(m_{z_0,w_0,\sigma}(z))\,|\,d\vol_H(z)
-
\int_{\Omega(z_0,R)}\,|\,\mu(m_{z_0,z'_0,1}z)-\nu(m_{z'_0,w'_0,\sigma}
\circ
 m_{z_0,z'_0,1}(z))\,|\,d\vol_H(u)\,\right|\nonumber\\
&~~~~~ \leq \int_{\Omega(z_0,R)}\,
\left(\,|\,\mu(z)-\mu(m_{z_0,z'_0,1}(z))\,| +
|\,\nu(m_{z_0,w_0,\sigma}(z))-\nu(m_{z'_0,w'_0,\sigma}(m_{z_0,z'_0,1}(z)))\,|\,\right)\,
d\vol_H(z) ~.\nonumber \label{interm2}
\end{align}
On the other hand, note that for any $\gamma >0$, $\mu$ and $\nu$
are continuous on the closures of $\Omega(z_0,R+\gamma)$ and
$\Omega(w_0,R+\gamma)$, respectively; since these closed hyperbolic
disks are compact, $\mu$ and $\nu$ are  bounded on these sets. Pick
now $\rho>0$ such that $|z'_0-z_0|<\rho$, $|w'_0-w_0|<\rho$ imply
that $\Omega(z'_0,R)\subset \Omega(z_0,R+\gamma)$ as well as
$\Omega(w'_0,R)\subset \Omega(w_0,R+\gamma)$. It follows that, if
$|z'_0-z_0|<\rho$ and $|w'_0-w_0|<\rho$, then
$|\,\mu(z)-\mu(m_{z_0,z'_0,1}(z)\,|$ and
$|\,\nu(m_{\z_0,w_0,\sigma}(z))-
\nu(m_{z'_0,w'_0,\sigma}(m_{z_0,z'_0,1}(z))) \,|$ are bounded
uniformly for $z \in \Omega(z_0,R)$. Since it is clear from the
explicit expressions (\ref{e:a_of_mobius}) that
$m_{z_0,z'_0,1}(z)\rightarrow z$ and
$m_{z'_0,w'_0,\sigma}(m_{z_0,z'_0,1}(z)) \rightarrow
m_{z_0,w_0,\sigma}(z)$ as $z'_0 \rightarrow z_0$ and $w'_0
\rightarrow w_0$, we can thus invoke the dominated convergence
theorem again to prove continuity of $\Phi(\cdot,\cdot,\sigma)$.

To prove the equicontinuity, we first note that $\nu$ is uniformly
continuous on $\Omega(w_0,R)\cup\Omega(w'_0,R) $, since $\nu$ is
continuous on the compact set $\overline{\Omega(w_0,R+\gamma)}$,
which contains $\Omega(w_0,R)\cup\Omega(w'_0,R)$ for all $w'_0$ that
satisfy $|w'_0-w_0| \leq \rho$. This means that, given any $\eps
>0$, we can find $\delta >0$ such that $|\nu(w)-\nu(w')|\le \eps$
holds for all $w,\,w'$ that satisfy $w,\,w' \in
\Omega(w_0,R)\cup\Omega(w'_0,R)$ and $|w-w'|\leq \delta$. This
implies the desired equicontinuity if we can show that
$|m_{z_0,w_0,\sigma}(z)-m_{z'_0,w'_0,\sigma}(m_{z_0,z'_0,1}(z))|$
can be made smaller than $\delta$, uniformly in $\sigma \in S_1$, by
making
$|z'_0-z_0|+|w'_0-w_0|$ sufficiently small.\\
We first estimate $|m_{z_0,w_0,\sigma}(z)-m_{z_0,w'_0,\sigma}(z)|$.
With the notations of (\ref{e:a_of_mobius}), we have
\begin{align*}
a(z_0,w_0,\sigma)-a(z_0,w'_0,\sigma)&=
\frac{(z_0-w_0\overline{\sigma})(1-\overline{z_0}w'_0\overline{\sigma})-
(z_0-w'_0\overline{\sigma})(1-\overline{z_0}w_0\overline{\sigma})}
{(1-\overline{z_0}w_0\overline{\sigma})(1-\overline{z_0}w'_0\overline{\sigma})}\\
&=\frac{(w_0-w'_0)\overline{\sigma}(|z_0|^2-1)}
{(1-\overline{z_0}w_0\overline{\sigma})(1-\overline{z_0}w'_0\overline{\sigma})}~,
\end{align*}
so that
\[
| a(z_0,w_0,\sigma)-a(z_0,w'_0,\sigma)|\leq \frac{|w_0-w'_0|}
{(1-|z_0|\,|w_0|)[1-|z_0|(|w_0|+\xi)]}\leq \frac{\xi}
{(1-|z_0|\,|w_0|)[1-|z_0|(|w_0|+\xi)]}
\]
when $|w_0-w'_0|<\xi$. It thus suffices to choose $\xi$ so that
$\xi< \zeta(1-|z_0|\,|w_0|)[1-|z_0|(|w_0|+\xi)]$ to ensure that $|
a(z_0,w_0,\sigma)-a(z_0,w'_0,\sigma)|<\zeta$. For the phase factor
$\tau$ in (\ref{e:a_of_mobius}) we obtain
\begin{align*}
 \tau(z_0,w_0,\sigma)-\tau(z_0,w'_0,\sigma)&=
\sigma \,\frac
{(1-\overline{z_0}w'_0\overline{\sigma})(1-z_0\overline{w_0}\sigma)
-
(1-\overline{z_0}w_0\overline{\sigma})(1-z_0\overline{w'_0}\sigma)}
{(1-\overline{z_0}w_0\overline{\sigma})(1-\overline{z_0}w'_0\overline{\sigma})}\\
&=\sigma \, \frac
{(w_0-w'_0)\overline{z_0}\overline{\sigma}-(\overline{w_0}-\overline{w'_0})z_0\sigma
+|z_0|^2(\overline{w_0}w'_0-\overline{w'_0}w_0)}
{(1-\overline{z_0}w_0\overline{\sigma})(1-\overline{z_0}w'_0\overline{\sigma})}\\
&=\sigma \, \frac {(w_0-w'_0)\overline{z_0}\overline{\sigma}-
z_0(\overline{w_0}-\overline{w'_0})\sigma +
|z_0|^2[\overline{w_0}(w'_0-w_0)+w_0(\overline{w_0}-\overline{w'_0}])
}
{(1-\overline{z_0}w_0\overline{\sigma})(1-\overline{z_0}w'_0\overline{\sigma})}~;
\end{align*}
when $\abs{w_0-w'_0}<\xi$, this implies
\[
|\tau(z_0,w_0,\sigma)-\tau(z_0,w'_0,\sigma)| \leq
\frac{|z_0|\,|w_0|\,[2+|z_0|(2|w_0|+\xi)]}
{(1-|z_0|\,|w_0|)[1-|z_0|(|w_0|+\xi)]}\xi~,
\]
which can clearly be made smaller than any $\zeta >0$ by choosing
$\xi$ sufficiently small. All this implies that (use
(\ref{e:a_of_mobius}))
\begin{align*}
|m_{z_0,w_0,\sigma}(z)-m_{z_0,w'_0,\sigma}(z)|
&\leq|\tau(z_0,w_0,\sigma)-\tau(z_0,w'_0,\sigma)|\frac{1+|z|}{1-|z|}
\,+\,|a(z_0,w_0,\sigma)-a(z_0,w'_0,\sigma)|\frac{(1+|z|)^2}
{(1-|z|)^2}\\
&\leq \zeta \,\frac{2(1+|z|)}{(1-|z|)^2},
\end{align*}
which will be  smaller than $\delta/2$, uniformly in $\sigma$, if
$\zeta < \delta (1-|z|^2)/8$; this bound on $\zeta$ in turn
determines the bound to be imposed on the $\xi$ used above. Hence
$|m_{z_0,w_0,\sigma}(z)-m_{z_0,w'_0,\sigma}(z)|<\delta/2$ can be
guaranteed, uniformly in $\sigma$, by choosing $|w_0-w'_0|<\xi$
for sufficiently small $\xi$. \\
One can estimate likewise
\[
|m_{z_0,w'_0,\sigma}(z)-m_{z'_0,w'_0,\sigma}(m_{z'_0,z_0,1}(z))|~,
\]
and show that this too can be made smaller than $\delta/2$,
uniformly in $\sigma$, by imposing sufficiently tight bounds on
$|z'_0-z_0|$ and $|w'_0-w_0|$. Combining all these estimates then
leads to the desired equicontinuity, as indicated earlier.
\end{proof}

To prove Lemma \ref{lem:extension_of_d}, we shall use the following
lemma:\\
\begin{lem}\label{lem:extermal_mobius}
Consider $u_k=e^{\bfi\psi} + \eps_k$, where $\abs{\eps_k}\too 0$ as
$k\too \infty$. Then there exists, for every $\eps>0$, a $K \in
\mathds{N}$ such that for all $k>K$; and all $\wh{m} \in
M_{D,0,u_k}$,
$$\inf_{w\in \Omega_{0,R}}\abs{\wh{m}(w)} > 1-\eps.$$
\end{lem}
The set $M_{D,0,u_k}$ used in this lemma is given by Definition
\ref{def:M_D,z_0,w_0}.
\begin{proof}
From Lemma \ref{lem:a_and_tet_formula_in_mobius_interpolation} we
can write $\wh{m}$ as $$\wh{m}(w)=e^{\bfi \theta}\frac{w+u_k
e^{-\bfi \theta}}{1 + \bbar{u_k}e^{\bfi \theta}w},$$ for some
$\theta \in [0,2\pi)$. Substituting $u_k = e^{\bfi \psi} + \eps_k$
in this equation we get
\begin{align*}
\wh{m}(w) & = e^{\bfi \theta}\frac{w+(e^{\bfi \psi} + \eps_k)
e^{-\bfi \theta}}{1 + \bbar{(e^{\bfi \psi} + \eps_k)}e^{\bfi
\theta}w}
 = e^{\bfi \psi} \frac{1+w e^{\bfi (\theta - \psi)} + \eps_k \eps^{-\bfi \psi}}{1 + w e^{\bfi (\theta - \psi)} + \bbar{\eps_k}e^{\bfi\theta}w }.
\end{align*}
Writing the shorthand $s$ for $s = 1+w e^{\bfi(\theta-\psi)}$, we
have thus
\begin{align*}
\abs{\wh{m}(w)-e^{\bfi \psi}} & = \abs{e^{\bfi \psi} \frac{s +
\eps_k \eps^{-\bfi \psi}}{s + \bbar{\eps_k}e^{\bfi\theta}w } -
e^{\bfi \psi}}
 \leq
\abs{ e^{\bfi \psi}\frac{\eps_k e^{-\bfi \psi} - \bbar{\eps}_k e^{\bfi \theta}w }{s + \bbar{\eps}_k e^{\bfi \theta }w }  } \\
& \leq \frac{\abs{ \eps_k e^{-\bfi \psi} - \bbar{\eps}_k e^{\bfi
\theta}w  }}{\abs{ s + \bbar{\eps}_k e^{\bfi \theta}w }}
 \leq
\frac{\abs{\eps_k}\parr{1+ |w|}}{|s|-|\eps_k||w|}
\end{align*}
Now for all $w\in \Omega_{0,R}$, $|w|<r_R=\tanh^{-1}(R)$. This
implies $|s| \geq 1 - |w| \geq 1-r_R$, and $1+|w| \leq 1+r_R$, so
that
$$\abs{\wh{m}(w)-e^{\bfi \psi}} \leq |\eps_k|\frac{1+r_R}{1-r_R - |\eps_k|r_R} = |\eps_k|\frac{1+r_R}{1-r_R\parr{1+ |\eps_k|}}.$$
Since $\abs{\eps_k}\too 0$ the lemma follows.
\end{proof}


We are now ready for \\
{\bf Lemma
\ref{lem:extension_of_d}} {\em Let $\set{(z_k,w_k)}_{k\geq 1} \subset
\D\times\D$ be a sequence that converges, in the Euclidean norm, to
some point $(z',w') \in \bbar{\D}\times\bbar{\D} \setminus
\D\times\D$, that is $|z_k-z'|+|w_k-w'| \too 0$, as $k \too \infty$.
Then, $\lim_{k\too \infty}d^R_{\xi,\zeta}(z_k,w_k)$ exists and
depends only on the limit point $(z',w')$.}
\begin{proof}
Since $(z',w') \in \bbar{\D}\times\bbar{\D} \setminus \D\times\D$
either $z'\in \bbar{\D}\setminus \D$ or $w'\in \bbar{\D}\setminus
\D$. Let us assume that $z'\in \bbar{\D}\setminus \D$ (the case
$w'\in \bbar{\D}\setminus \D$ is similar). Denote by $m_k$ an
arbitrary M\"{o}bius transformation in $M_{D,0,w_k}$. By symmetry of
the distance and using a change of variables we then obtain
\begin{align*}
d^R_{\xi,\zeta}(z_k,w_k) & = d^R_{\zeta,\xi}(w_k,z_k) \\
& = \min_{m(w_k)=z_k} \int_{\Omega_{w_k,R}} \Big |
\zeta(w)-\xi(m(w)) \Big |d\vol_H(w) \\ & = \min_{m(w_k)=z_k}
\int_{\Omega_{0,R}}\Big | \zeta(m_k(w))-\xi(m(m_k(w))) \Big
|d\vol_H(w).
\end{align*}
Now, recall that $\xi(z) = \xi^H(z) = \wt{\xi}(z) (1-|z|^2)^2$,
where $\wt{\xi}(z)$ is a bounded function, $\sup_{z\in
\D}|\wt{\xi}(z)| \leq C_{\wt{\xi}}$. From Lemma
\ref{lem:extermal_mobius} we know that for every $\eps>0$ and for
$k>K$ sufficiently large, $\abs{m(m_k((w))} > 1 -\eps$ for all $w\in
\Omega_{0,R}$, and all $m$ such that $m(w_k)=z_k$. This means that
for these $k>K$ we have
\begin{align*}
\abs{\xi(m(m_k(w)))} & =
\abs{\wt{\xi}(m(m_k(w)))}(1-\abs{m(m_k(w))}^2)^2 \\ \ & \leq
C_{\wt{\xi}}(1-(1-\eps)^2)^2\leq C_{\wt{\xi}}\eps^2(2-\eps)^2,
\end{align*}
for all $w \in \Omega_{0,R}.$
Therefore,
\begin{align*}
& \abs{d^R_{\xi,\zeta}(z_k,w_k) -
\int_{\Omega_{0,R}}\babs{\zeta(m_k(w))}d\vol_H(w)} \\ & \leq \abs{
\min_{m(w_k)=z_k}
\int_{\Omega_{0,R}}\babs{\zeta(m_k(w))-\xi(m(m_k(w)))}d\vol_H(w) -
\int_{\Omega_{0,R}}\babs{\zeta(m_k(w))}d\vol_H(w) } \\ &
 \leq
\abs{ \min_{m(w_k)=z_k} \int_{\Omega_{0,R}} \set{
\babs{\zeta(m_k(w))-\xi(m(m_k(w)))} - \babs{\zeta(m_k(w))} }
d\vol_H(w)} \\ & \leq \min_{m(w_k)=z_k} \int_{\Omega_{0,R}}
\babs{\xi(m(m_k(w)))} d\vol_H(w) \too 0, \ \mathrm{as} \ k\too
\infty.
\end{align*}
Therefore $d^R_{\xi,\zeta}(z_k,w_k)$ converges, as $k\too \infty$,
if and only if $\int_{\Omega_{0,R}}\abs{\zeta(m_k(w))}d\vol_H(w)$
converges, and to the same limit, for any $m_k \in M_{D,0,w_k}$. We
can take, for instance, $m_k(w) = \frac{w+w_k}{1+\bbar{w_k}w}$ which
gives
$$\int_{\Omega_{0,R}}\abs{\zeta(m_k(w))}d\vol_H(w)  =
\int_{\Omega_{0,R}}\abs{\zeta\parr{\frac{w+w_k}{1+\bbar{w_k}w}}}d\vol_H(w).$$
For $w\in \Omega_{0,R}$, $\abs{1+\bbar{w_k}w} > 1 - r_R$. It follows
that this expression has a limit as $k \too \infty$, and
$$\lim_{k\too \infty} \int_{\Omega_{0,R}}\abs{\zeta(m_k(w))}d\vol_H(w) =
\int_{\Omega_{0,R}}\abs{\zeta\parr{
\frac{w+w'}{1+\bbar{w'}w}}}d\vol_H(w),$$ which clearly depends on
$w'$, not on the sequence $\set{w_k}$.
\end{proof}

\end{document}